\documentclass[11pt]{amsart}
\usepackage{amssymb,amsfonts, amsthm, amsmath}
\usepackage{etaremune}
\usepackage{enumerate}
\usepackage{xcolor}
\usepackage[utf8]{inputenc}
\usepackage{color}
\usepackage{comment}
\usepackage{hyperref}
\usepackage[symbol]{footmisc}
\usepackage{mathtools}
\usepackage[margin=1.4in]{geometry}
\usepackage[normalem]{ulem}
\usepackage{mathrsfs}
\newcommand{\Balpha}{\mbox{$\hspace{0.1em}\rule[0.01em]{0.05em}{0.39em}\hspace{-0.21em}\alpha$}}

\newtheorem{theorem}{Theorem}[section]
\newtheorem*{theorem*}{Theorem}
\newtheorem{corollary}[theorem]{Corollary}
\newtheorem{prop}[theorem]{Proposition}
\newtheorem{lemma}[theorem]{Lemma}
\theoremstyle{definition}
\newtheorem{defn}[theorem]{Definition}
\newtheorem{remark}[theorem]{Remark}
\newtheorem{example}[theorem]{Example}
\newtheorem{claim}{Claim}[section]

\newcommand{\hneck}{H_\mathrm{neck}}
\newcommand{\hthic}{H_\mathrm{th}}
\newcommand{\htrig}{H_\mathrm{trig}}
\newcommand{\surg}{\mathcal{M}_\mathbb{H}}
\newcommand{\surgi}{\mathcal{M}_{\mathbb{H}_i}}

\title[Approximation of mean curvature flow with generic singularities]{Approximation of mean curvature flow with generic singularities by smooth flows with surgery}
\author{J. M. Daniels-Holgate}
\address{Department of Mathematics, Zeeman Building, University of Warwick, Gibbet Hill Road, Coventry CV4 7AL,
UK}
\email{joshua.daniels-holgate@warwick.ac.uk} 
\date{October 2022}

\begin{document}
\maketitle

\begin{abstract}
    We construct smooth mean curvature flows with surgery that approximate weak mean curvature flows with only spherical and neck-pinch singularities. This is achieved by combining the recent work of Choi-Haslhofer-Hershkovits, and Choi-Haslhofer-Hershkovits-White, establishing canonical neighbourhoods of such singularities, with suitable barriers to flows with surgery. A limiting argument is then used to control these approximating flows. We conclude by improving the entropy bound on the low-entropy Schoenflies conjecture.
\end{abstract}

\section{Introduction}
 Mean curvature flow is the $L^2$-gradient flow for the area functional. In general, the flow from a hypersurface can develop singularities and there are multiple notions of weak flow that allow for the continuation of the flow past such singularities. An alternate approach is to approximate the flow by a piece-wise smooth flow, known as a mean curvature flow with surgery. The surgery procedure for mean curvature flow from a 2-convex hypersurface of dimension $n\geq3$ was introduced by Huisken--Sinestrari in \cite{hssurg}, and extended to $n=2$ by Huisken--Brendle \cite{HBsurg}. Independently, Haslhofer--Kleiner \cite{hkestimates, hksurg} established a surgery procedure that works for all dimensions $n\geq2$. By classifying blow ups for a more general class of 2-convex flows, they showed regions of high curvature in such flows have a canonical structure.

In both methodologies, existence of 2-convex surgery boils down to the classification of regions of high curvature that develop: a canonical neighbourhood theorem for 2-convex flow. Canonical neighbourhoods of neck-pinch singularities for unit-regular cyclic (mod 2) Brakke flows of dimension $n=2$ were established in \cite{chh18} and for $n\geq3$ in \cite{chhw19}, as a corollary to their resolution of the mean convex neighbourhood conjecture for neck-pinch singularities. It is from this result that we can extend the smooth mean curvature flow with surgery.

Spherical and generalised cylindrical singularities were conjectured by Huisken to be `generic', \cite[\#~8]{Trieste}. The pioneering work of Colding--Minicozzi, \cite{CM12,CM15, CM16}, catalyzed the study of generic flows through their introduction of the entropy functional and establishing of Łojasiewicz-type inequalities. Further, they showed spherical and generalised cylindrical singularities are the only linearly stable singularity models. The study of generic flows has recently been furthered by the work of Chodosh--Choi--Mantoulidis--Schulze, \cite{ccms20, ccms21}. They showed that hypersurfaces in $\mathbb{R}^4$ with entropy less than that of $\mathbb{S}^1\times\mathbb{R}^2$ can be perturbed such that the weak flow from this perturbed surface encounters only spherical and neck-pinch singularities. Such results provide a strong motivation for establishing a flow with surgery. Recall, a flow with surgery will have finitely many surgeries. This provides a simple way for topological information to be tracked. See Section 6, where we prove the low-entropy Schoenflies conjecture \cite[Conjecture 1.9]{ccms21} in such a manner. Indeed finiteness is desirable, as despite the groundbreaking results concerning the structure and size of the singular set, see
White \cite{bw97} and Colding--Minicozzi \cite{CM15}, it is still unknown if there are finitely many singular times, or if spherical singularities can accumulate to a neck-pinch singularity. See the work of B.Choi--Haslhofer--Hershkovits \cite{bchh}.

To highlight why existence of a surgical flow is non-trivial, consider a hypersurface, $M$, whose mean curvature flow has only spherical and neck-pinch singularities, and a single (non-degenerate) neck-pinch singularity at the first singular time. With the canonical neighbourhood theorems of \cite{chh18, chhw19} in mind, one can follow the arguments of \cite{hksurg} to pick surgery parameters suitable for surgical modifications to be made at some time before the flow becomes singular. Such a process would construct a new hypersurface $M'$. One immediately runs into a problem: without assuming global 2-convexity, we do not have any knowledge of how the flow from $M'$ will proceed. In the worst case, it may run into non-generic singularities. Moreover, the concatenation of these flows is no longer a weak flow, so passing to global limits along sequences of modified flows becomes impractical. To overcome these difficulties, we develop a technical framework that allows us to pass to limits locally. Further, we show the flows converge, in a smooth sense, to the original weak flow. This gains control of the flows with surgical modification, allowing for one to perform subsequent surgeries.

\subsection{Overview}
We adapt the definitions of \cite{hksurg} to construct a unit-regular Brakke flow with surgical modification. This gives one the freedom to localise the surgery.\footnote{Ultimately, one will use the maximum principle to show the existence arguments can be applied directly. There is no reason that the formalism of \cite{hssurg} and \cite{HBsurg} could not be used, however, the formalism of \cite{hksurg} makes it very clear what data one has to control on the boundary.}

Throughout this work, we will be considering an $n$-dimensional unit-regular, cyclic (mod 2) integral Brakke flow $\mathcal{M}$ that encounters only spherical or neck-pinch singularities (with multiplicity one), evolving from the smoothly embedded, closed hypersurface $M^n\subset\mathbb{R}^{n+1}$. We recall the definition of such singularities.
\begin{defn} A (multiplicity-one) singularity is said to be 
\begin{enumerate}[(a)]
    \item \textit{spherical} if it has the shrinking sphere $
    (-\infty , 0) \ni t \mapsto \mathbb{S}^{n}(\sqrt{-2nt})\times \mathbb{R}$ as a tangent flow
    \item a \textit{neck-pinch} if it has the shrinking cylinder $
    (-\infty , 0) \ni t \mapsto \mathbb{S}^{n-1}(\sqrt{-2(n-1)t})\times \mathbb{R}$ as a tangent flow.
\end{enumerate}
\end{defn} By the work of Hershkovits-White \cite{hershwhite}, and the resolution of the mean convex neighbourhood conjecture, a level set flow with only these singularities does not fatten. Moreover, these results, plus the recent work \cite{ccms21}, provide the tools required to prove a uniqueness theorem for weak mean curvature flows with only spherical and neck-pinch singularities. In Theorem \ref{bigunique}, we show that if the outer flow from a given hypersurface $M^n\subset\mathbb{R}^{n+1}$ encounters only spherical and neck-pinch singularities, then it is the unique, unit-regular, cyclic (mod 2), integral Brakke flow starting from $M$. For readers unfamiliar with such terminology, we refer to Section \ref{prelim}.

Our principal result concerns the existence of a smooth flow with surgery from a given hypersurface.
  
  The existence of a surgery flow is dependent on two parameters, $H_\mathrm{min}$ and $\Theta$. Recall, the parameters of surgery detailed in \cite{hksurg} are: $H_\mathrm{th}$, the scale at which components are dropped, $H_\mathrm{neck}$, the scale of the necks which we perform surgery on, and $H_\mathrm{trig}$, the trigger scale, at which we pause the flow and perform surgery. The parameter $\Theta$ governs the ratios between these quantities. We say $\mathbb{H}\geq \Theta$ if $H_\mathrm{trig}/H_\mathrm{neck}\geq \Theta$ and $H_\mathrm{neck}/H_\mathrm{th}\geq \Theta$. We also require $\hthic> H_\mathrm{min}$.
\begin{theorem}[Existence of a smooth flow with surgery]
Let $M^n\subset\mathbb{R}^{n+1}$ be a smoothly embedded hypersurface, and $\mathcal{M}$ be a unit-regular, cyclic mod 2 integral Brakke flow, emerging from $M$ with only spherical and neck-pinch singularities.
Then, the parameters $H_\mathrm{min}(M)<\infty$ and $\Theta(M)<\infty$ can be chosen (depending only on the initial hypersurface) such that every weak $(\Balpha,\delta,\mathbb{H})$-flow, $\mathcal{M}_\mathbb{H}$, with $H_\mathrm{th}>H_\mathrm{min}$, $\mathbb{H}>\Theta$ satisfies:
 \begin{itemize}
     \item $|H|\leq H_\mathrm{trig}<\infty$ everywhere, 
     \item $\mathcal{M}_\mathbb{H}$ vanishes in finite time.
 \end{itemize}
 i.e.~ $\mathcal{M}_\mathbb{H}$ is a smooth mean curvature flow with surgery.
 \end{theorem}
 
 For the precise definition of a weak $(\Balpha,\delta,\mathbb{H})$-flow, see Definition \ref{weakflow}.

Our proof relies on two key ideas. The first is the construction of barriers to flow with surgery, Theorem \ref{barriers}, to establish Hausdorff convergence of surgical flows to the level set flow. Such an idea was first explored by Lauer \cite{lauer} for 2-convex flows. Their idea is not directly applicable, as they take advantage of the set monotonicity of such flows. Instead, we consider flows from near-by initial conditions and show they act as barriers to surgery flows.

  Before detailing the second tool, we make the following observations. Let $\{\mathcal{N}^i\}_{i\in\mathbb{N}}$ be sequence of integral unit-regular Brakke flows, and presume each flow has a singular set of small Hausdorff dimension. Suppose the sequence converges in the Hausdorff sense to a Brakke flow $\mathcal{M}$. By further assuming $\mathcal{N}^i$ converge smoothly to $\mathcal{M}$ at the initial time, the result of \cite{ccms20} allows for Hausdorff convergence to be improved to Brakke convergence. Turning our attention back to weak flows with surgery, we observe in regions where no surgical modifications take place, a surgical flow is a smooth mean curvature flow. It is hence desirable to understand where surgical modifications take place. This is the purpose of our second tool, Proposition \ref{surgcontrol}, which shows surgeries accumulate in the singular set. Moreover, we actually show the smooth convergence of the flows with surgery by probing the behaviour of flows with surgery in neighbourhoods of regular points of $\mathcal{M}$ with a careful combination of pseudolocality for mean curvature flow \cite{pseudo}, graphical estimates \cite{EHest} and the curvature estimates of Haslhofer--Kleiner, \cite{hksurg}. This second tool requires us to only permit surgery in a set with somewhat technical restrictions on the behaviour of the flow along the boundary. These requirements ensure that the hypotheses of the curvature estimates are satisfied.

We consider $\Omega_{(\alpha,\beta)}$ - an open neighbourhood of the singular set with finitely many connected components, along the boundary of which the flow $\mathcal{M}$ behaves in a fashion suitable for surgery in the interior. We examine the class of weak flows with surgery, derived from $M$. Surgeries are performed only in the set $\Omega_{(\alpha,\beta)}$.

As previously noted, a priori little can be known about the long time behaviour of modified flows due to the parabolic nature of mean curvature flow. Using the above tools we demonstrate the parameters can be chosen suitably such that the surgery flow will be a small graph over $\mathcal{M}$ along the boundary of $\Omega_{(\alpha,\beta)}$.  The existence of suitable parameters is shown by a convergence result, Proposition \ref{convergence2}.  It then follows that the weak surgery flows are smooth flows with surgery in the sense of Haslhofer--Kleiner inside $\Omega_{(\alpha,\beta)}$, the canonical neighbourhoods of the flow $\mathcal{M}$, via the maximum principle, and hence the arguments of Haslhofer--Kleiner can be applied to show the existence of a smooth flow with surgery.

In addition, we show that such mean curvature flows with surgery approximate the weak flow, compare \cite{lauer, head} in the 2-convex case.
 
 \begin{theorem}
 Taking the limit as $H_\mathrm{th}\to \infty$, the weak $(\Balpha,\delta,\mathbb{H})$ surgical flows converge in the Hausdorff sense to $\mathcal{M}$. In particular, away from the singular set of $\mathcal{M}$ the convergence is smooth.
 \end{theorem}
 
 Finally, we combine our proof of the existence of a mean curvature flow with surgery with the existence of generic low entropy flows established by Chodosh--Choi--Mantoulidis--Schulze to get a new bound on entropy for the low-entropy Schoenflies conjecture, as conjectured in \cite[Conjecture 1.9]{ccms21}.
 \begin{theorem}[Low-entropy Schoenflies for $\mathbb{R}^4$]
    Let $\Sigma^3\subset\mathbb{R}^4$ be a hypersurface homeomorphic to $\mathbb{S}^3$ with entropy $\lambda(\Sigma)\leq\lambda(\mathbb{S}^1\times \mathbb{R}^2)$. Then $M$ is smoothly isotopic to the round $\mathbb{S}^3$.
 \end{theorem}
Surgery is used to decompose the surface into spheres and tori, and the topological properties of the flow are exploited to rule out tori. The previous best bound was established independently by Bernstein--Wang \cite{bw20} and Chodosh--Choi--Mantoulidis--Schulze \cite{ccms20}.
 
\subsection{Organisation}
 In Section 2, we recap the structure of Haslhofer--Kleiner surgery. In Section 3, we discuss the necessary adaptations to the definitions of \cite{hksurg} for our more general setting. In Section 4, we construct barriers and detail the structure and stability of weak surgery flows. In Section 5, we prove the existence of a smooth mean curvature flow with surgery approximating the unit-regular Brakke flow. Finally, in Section 6 we apply the results to the low-entropy Schoenflies conjecture. 
 
 \subsection{Acknowledgements}
A great deal of thanks goes to Felix Schulze, the author's supervisor, for the discussion and guidance provided. The author is also grateful to Otis Chodosh and Huy The Nguyen. The author would like to thank the referee for their constructive comments.

\section{Preliminaries}\label{prelim}
For the convenience of the reader, we re-state central definitions and tools from the field.

\begin{defn}
The parabolic cylinder of radius $r>0$ centred at the space-time point $X=(\mathbf{x},t)\in\mathbb{R}^{n+1}\times\mathbb{R}$ is defined as
\begin{align*}
    P(X,r)=B(\mathbf{x},r)\times (t-r^2,t+r^2)
\end{align*}
We use the terminology `backwards (resp. forwards) parabolic cylinder' for a parabolic cylinder with a time interval of the form $(t-r^2,t]$, (resp. $ [t,t+r^2)$).
\end{defn}
\begin{defn}[Mean Curvature flow] Let $M^n \subset \mathbb{R}^{n+1}$ be a smoothly embedded hypersurface. A mean curvature flow
$\mathcal{M}=\{M_t\subset U\}_{t\in[0,t_0)}$ in an open subset $U\subset \mathbb{R}^{n+1}$ is a smooth family of hypersurfaces such that \begin{align*}
    M_0 &= M, \\
    \left(\frac{\partial}{\partial t}\mathbf{x}\right)^\perp&=\mathbf{H}_{M_t}(\mathbf{x})\, ,
\end{align*}
where $\mathbf{H}_{M_t}(\mathbf{x})$ is the mean curvature vector.
\end{defn}
\begin{defn}
Given a choice of unit normal, $\nu$, we fix an orientation, and thus can write
\begin{align*}
    \mathbf{H}= - H\nu
\end{align*}
We refer to $H=H(\mathbf{x})$ as the (scalar) mean curvature.
\end{defn}
The flow is non-linear and develops singularities. A rich theory has been developed to continue the flow past such singularities. 
\begin{defn}[Integral Brakke Flow \cite{Brakke,ilmanen}]
We follow the formalism of \cite{white21}.
An ($n$-dimensional) \textit{integral Brakke flow} in $\mathbb{R}^{n+1}$ is a 1-parameter family of Radon measures $\{\mu_t\}_{t\in I}$ over an interval $I\subset \mathbb{R}$ such that:
\begin{enumerate}[(i)]
   \item For almost every $t$ there exists and integral $n$-dimensional varifold $V(t)$ with $\mu_t=\mu_{V(t)}$ so that $V(t)$ has locally bounded first variation and has mean curvature $\mathbf{H}$ orthogonal to $\mathrm{Tan}(V(t),\cdot)$ almost everywhere.
    \item For a bounded interval $[t_1,t_2]\subset I$ and any compact set $K$ \begin{align*}
        \int^{t_2}_{t_1}\int_K(1+|\mathbf{H}|^2)\,d\mu_t\,dt<\infty\, .
    \end{align*}
    \item If $[t_1,t_2]\subset I$ and $f\in C^1_c(\mathbb{R}^{n+1}\times [t_1,t_2])$ has $f\geq0$ then \begin{align*}
        \int f(\cdot,t_2)\,d\mu_{t_2}-\int f(\cdot,t_1)\,d\mu_{t_1}\leq \int^{t_2}_{t_1}\int_K \Big(-|\mathbf{H}|^2f+\mathbf{H}\cdot\nabla f +\frac{\partial}{\partial t}f\Big)\,d\mu_t\, dt
    \end{align*}
\end{enumerate}
We write $\mathcal{M}$ for a Brakke flow $\{\mu_t\}_{t\in I}$ to refer to the family of measures $I\ni t\mapsto \mu_t $ satisfying Brakke's inequality.
\end{defn}

\begin{defn}[Density and Huisken's Monotonicity]
For $X_0:=(\mathbf{x}_0,t_0)\in \mathbb{R}^{n+1}\times \mathbb{R}$, consider the backward heat kernel based at $(\mathbf{x}_0,t_0)$:
\begin{align*}
    \rho_{X_0}(\mathbf{x},t)=(4\pi(t_0-t))^{-n/2}\exp\left(-\frac{|\mathbf{x}-\mathbf{x}_0|^2}{4(t_0-t)}\right),
\end{align*}
for $\mathbf{x}\in\mathbb{R}^{n+1}, t<t_0$. For a Brakke flow $\mathcal{M}$ and $r>0$ we set
\begin{align*}
    \Theta_\mathcal{M}(X_0,r):=\int_{\mathbb{R}^{n+1}}\rho_{X_0}(\mathbf{x},t_0-r^2)\, d \mu_{t_0-r^2}
\end{align*}
$\Theta_\mathcal{M}(X_0,r)$ is known as the density ratio at $X_0$ at scale $r>0$. Huisken's monotonicity formula \cite{huis90} implies that 
\begin{align*}
    \frac{d}{dt}\int \rho_{X_0}(\mathbf{x},t)\, d\mu_t\leq -\int \left|\mathbf{H}-\frac{(\mathbf{x}-\mathbf{x}_0)^\perp}{2(t-t_0)}\right|^2\rho_{X_0}(\mathbf{x},t)\, d\mu_t\, .
    \end{align*}
  In particular, the Gaussian density of $\mathcal{M}$ at $X_0$ is defined by
    \begin{align*}
        \Theta_\mathcal{M}(X_0):=\lim_{r\searrow0}\Theta_\mathcal{M}(X_0,r)\, .
    \end{align*}
\end{defn}
\begin{defn}[Parabolic Rescaling]
Let $\mathcal{M}=\{M_t\}_{t\in[0,T)}$ be a mean curvature flow (Brakke flow). For any $\lambda>0$, we denote the parabolic rescaling of space-time by $\lambda$ as $\mathcal{D}_\lambda:(\mathbf{x},t)\mapsto (\lambda \mathbf{x}, \lambda^2 t)$.
We denote by $\mathcal{D}_\lambda(\mathcal{M}-X_0)$ the mean curvature flow (resp. Brakke flow) obtained from $\mathcal{M}$ by parabolic dilation around $X_0$ by $\lambda$. That is,
\begin{align*}
    \mathcal{D}_\lambda(\mathcal{M}-X_0)=\{ \mu^\lambda_t \}_{t'\in [-\lambda^2t_0,\lambda^2(T-t_0))},\\
    \mathrm{with }\ \mu^\lambda_t(A)=\lambda^n \mu_{t_0+\lambda^{-2}t}(\lambda^{-1}A+x_0) 
\end{align*}
\end{defn}
\begin{defn}[Tangent flow]
Let $\{\lambda_i\}$ be a sequence s.t. $\lambda_i\to\infty$. We define a tangent flow at the space-time point $X_0\in \mathcal{M}$ as a subsequential limiting Brakke flow of the sequence parabolic rescalings of $\mathcal{M}$ around $X_0$ by $\lambda_i$.
\end{defn}

The monotonicity formula implies that all tangent flows are self-similar, i.e.~their time $-1$ slice is given by a (weak) self-shrinker. 
\begin{defn}[Self-shrinker]
A hypersurface $\Sigma\subset\mathbb{R}^{n+1}$ is called a \textit{self-shrinker} if
\begin{align*}
    \mathbf{H}_\Sigma(\mathbf{x})+\frac{\mathbf{x}^\perp}{2} = 0.
\end{align*}
\end{defn}

We will only be considering Brakke flows with
(a) Spherical and (b) Neck-pinch singularities.

\begin{remark}
  Tangent flows are not necessarily unique, however, it follows from \cite{Huisken84} that at a point with a multiplicity one spherical tangent flow, all tangent flows are spheres. For multiplicity one cylindrical tangent flows, uniqueness was established in \cite{CM15}, so the above tangent flows are unique, and one can refer to \textit{the} tangent flow.
 \end{remark}
 \begin{remark} \label{singularstructure}
 The structure of the singular set of a Brakke flow $\mathcal{M}$ with spherical and (generalised) cylindrical singularities is well understood, see \cite{bw97, CM15, CM16}.
 
 \end{remark}

We will be considering unit-regular and cyclic (mod 2)  Brakke flows. The definition of an integral Brakke flow permits sudden vanishing; to (partially) avoid this, one can define the class of unit-regular Brakke flows. This class forbids vanishing at regular points of the flow.
\begin{defn}[Unit-regular and cyclic Brakke Flows \cite{White09}] 
An integral Brakke flow $\mathcal{M}=\{\mu_t\}_{t\in I}$ is said to be
\begin{itemize}
   
\item \textit{unit-regular} if $\mathcal{M}$ is smooth in some space-time neighbourhood of any space-time point $X$ with $\Theta_\mathcal{M}(X)=1$;
\item \textit{cyclic (mod 2)} if, for a.e.~ $t\in I, \mu_t=\mu_{V(t)}$ for an integral varifold $V(t)$ whose unique associated rectifiable mod-2 flat chain $[V(t)]$ has $\partial[V(t)]=0$.
\end{itemize}
\end{defn}

Finally, we state the following theorem from \cite{ccms20}. The ideas will be used in Section \ref{barriersandstability} to show convergence properties of the $\varepsilon$-barriers and of flows with surgery.

\begin{defn}
For a Brakke flow $\mathcal{M}$, we define $\widehat{\mathrm{reg}}\, \mathcal{M}$ to be the set of points $X=(\mathbf{x},t)$ such that there is an $\varepsilon>0$ with \begin{align*}
    \mathcal{M}\lfloor(B_\varepsilon(\mathbf{x})\times(t-\varepsilon^2,t]=k\mathcal{H}^n\lfloor M(t),
\end{align*}
where $k$ is a positive integer and $M(t)$ is a smooth mean curvature flow. 
We write $\mathrm{reg}\,\mathcal{M}$ as the above set with $k=1$; thus, $\mathrm{reg}\,\mathcal{M}\subset\widehat{\mathrm{reg}}\, \mathcal{M}$.
\end{defn}

\begin{theorem}[\textrm{\cite[Corollary F.4]{ccms20}}]\label{smallsing}
Suppose that $\mathcal{M}$ is a unit-regular integral $n$-dimensional Brakke flow in $\mathbb{R}^{n+k}$ with $\mu(t) = \mathcal{H}^n\lfloor M(t)$ for $t \in [0,\delta)$, where $M(t)$ is a mean curvature flow of connected, properly embedded submanifolds of $\mathbb{R}^{n+k}$ and $\delta>0$. If
\begin{align*}
    \mathcal{H}^n_P(\mathrm{supp}(\mathcal{M})\backslash\widehat{\mathrm{reg}}\mathcal{M})=0
\end{align*}
Then $ \widehat{\mathrm{reg}}\, \mathcal{M}=\mathrm{reg}\,\mathcal{M}$ is connected. 
\end{theorem}

 Here $\mathcal{H}^n_P$ denotes $n$-dimensional parabolic Hausdorff measure. This theorem provides vital information on the behaviour of unit-regular Brakke flows with small singular set.

Another formulation of a weak solution to the mean curvature flow is that of the level set flow. It was first introduced as a viscosity solution to the mean curvature flow independently by Evans--Spruck \cite{ES1} and Chen--Giga--Goto \cite{cgg}. The following geometric definition was given by Ilmanen, \cite{ilmanen}.

\begin{defn}[Weak and Level set flow, \cite{ilmanen}]
Let $K\subset\mathbb{R}^{n+1}$ be closed. A one-parameter family of closed sets, $\{K_t\}_{t\geq0}$, with initial condition $K_0=K$ is said to be a \textit{weak set flow} for $K$ if for every smooth mean curvature flow $M_t$ of compact hypersurfaces defined on $[t_0,t_1]$, we have 
\begin{align*}
    K_{t_0}\cap M_{t_0}=\emptyset \implies K_t\cap M_t=\emptyset
\end{align*}
for all $t\in [t_0,t_1]$.

The \textit{level set flow} is defined as the maximal weak set flow, i.e.~ the union of all weak set flows from $K$.
\end{defn}

\subsection{Overview of 2-Convex Surgery}
 The following is a recap of \cite{hksurg}. 
 \begin{defn}[$\alpha$-noncollapsed, \cite{andrews}, \cite{hkestimates}]
Let $\alpha>0$. A mean convex hypersurface $M^n$ bounding an open region $\Omega$ in $\mathbb{R}^{n+1}$
is $\alpha$–noncollapsed (on the scale of the mean curvature) if for every $x\in M $ there are closed
 balls $B_\mathrm{int}\subset \overline\Omega$ and $B_\mathrm{ext}\subset \mathbb{R}^{n+1}\backslash\Omega$ of radius at least $\alpha/H(x)$ tangential to $M$ at $x$, from the interior and exterior of $M$ respectively. A smooth mean curvature flow is said to be $\alpha$-noncollapsed if every time slice is $\alpha$-noncollapsed.
\end{defn} 

This definition may be suitably localised. See Definition \ref{localandrews}.
\begin{defn}[$\beta$-uniformly 2-convex]
A mean convex hypersurface $M$ is said to be $\beta$-uniformly 2-convex, for $\beta>0$, if 
\begin{align*}
    \lambda_1+\lambda_2>\beta H.
\end{align*}
Where $\lambda_i$ are the ordered principal curvatures with $\lambda_1\leq\ldots\leq\lambda_n$, and $H$ is the mean curvature.
\end{defn}

Recall, `$\alpha$-noncollapsed'-ness is preserved under the mean curvature flow by the maximum principle, \cite{andrews}. $\beta$-uniform 2-convexity is preserved by the Hamilton tensor maximum principle.

\begin{defn}[Strong $\delta$-neck \text{\cite[Definition 2.3]{hksurg}}]
Let $\delta>0$. We say a mean curvature flow $\mathcal{M}=\{M_t\subset U\}_{t\in I}$ has a strong $\delta$-neck with centre $p$ and radius $s$ at time $t_0\in I$ if $\mathcal{M}_{(p,t_0),s^{-1}}=\mathcal{D}_{s^{-1}}(\mathcal{M}-(p,t_0))$ is $\delta$-close in $C^{\lfloor1/\delta \rfloor}$ in $(B^U_{1/\delta}\times (-1,0])$ to the evolution of a solid round cylinder of radius 1 at $t=0$. Here $B^U_{1/\delta}= s^{-1}((B(p,s/\delta)\cap U)-p)\subseteq B(0,1/\delta)\subset \mathbb{R}^{n+1}$ and $\mathcal{D}_\lambda$ denotes the parabolic dilation by $\lambda$.
\end{defn}
\begin{defn}[Standard cap \text{\cite[Definition 2.2]{hksurg}}]\label{stdcap}
A standard cap is a smooth convex domain $K^{st} \subset \mathbb{R}^{n+1}$ that coincides with a solid
round half-cylinder of radius 1 outside a ball of radius 10.
\end{defn}
The evolution from such a cap is unique, $\beta$-uniformly 2-convex and $\alpha$-noncollapsed for some $\alpha,\beta>0$, \cite[Proposition 3.8]{hksurg}. This is a key component of the canonical neighbourhood theorem for mean curvature flows with surgery.

A surgery algorithm seeks to replace $\delta$-necks with standard caps, the following is the gluing algorithm used.

\begin{defn}[Replacing a $\delta$-neck by standard caps \text{\cite[Definition 2.4]{hksurg}}]\label{necktocap}
We say that the final time slice of a strong $\delta$-neck with centre $p$ and radius $s$ is replaced by a pair of standard caps if the pre-surgery domain $K^- \subset U$ is replaced by a post-surgery domain $ K^{\#} \subset K^-$ such that following statements hold.
\begin{enumerate}
    \item The modification takes place inside a ball $B=B(p,5\Gamma s)$
    \item There are bounds for the second fundamental form and its derivatives \begin{align*}
        \sup_{\partial K^\#\cap B} |\nabla^\ell A|\leq C_\ell s^{-1-\ell}
    \end{align*}
    \item If $B$ from point (1) satisfies $B \subset U$ then for every point $p_\# \in \partial K^\# \cap B$ with $\lambda_1(p_\#)<0$ there is a point $p_{-}\in \partial K^-\cap B$ with $\frac{\lambda_1}{H}(p_-)\leq\frac{\lambda_1}{H}(p_\#)$ 
    \item If $B(p, 10\Gamma s)\subset U$ then $s^{-1}(K^\#-p))$ is $\delta$-close in $B(0, 10 \Gamma)$ to a pair of disjoint standard caps which are at distance $\Gamma$ from the origin.
\end{enumerate}
Here, $\Gamma >0$ denotes a cap separation parameter that is fixed later.
\end{defn}

Haslhofer--Kleiner begin by defining a broader class of flows, of which mean curvature flow with surgery belongs. It is a class of piece-wise smooth, mean convex, $\alpha$-noncollapsed, mean curvature flows with $\delta$-necks replaced by caps. They fix a $\mu\in[1,\infty)$, used below.
\begin{defn}[$(\alpha,\delta)$-flow \text{\cite[Definition 1.3]{hksurg}}]\label{alphadelta}
An $(\alpha,\delta)$-flow $\mathcal{K}$ is a collection of finitely many smooth $\alpha$-noncollapsed flows $\{K^i_t\subset U\}_{t\in[t_{i-1},t_i]}, \ (i=1,\ldots, k; \ t_0<\cdots,t_k)$ in an open set $U\subset\mathbb{R}^{{n+1}}$ such that the following statements hold.
\begin{enumerate}
    \item For each $i=1,\ldots, k-1$, the final time slices of some collection of disjoint strong $\delta$-necks are replaced by pairs of standard caps as described in definition \ref{necktocap}, giving a domain $K^\#_{t_i}\subseteq K^i_{t_i}=:K^-_{t_i}$
    \item The initial time slice of the next flow $K^{i+1}_{t_i}=:K^+_{t_i}$, is obtained from $K^\#_{t_i}$ by discarding some connected components.
    \item There exists $s_\#=s_\#(\mathcal{K})>0$, which depends on $\mathcal{K}$, such that all necks in item (1) have radius $s\in [\mu^{-1/2}s_\#,\mu^{1/2}s_\#]$.
\end{enumerate}
\end{defn}

\begin{prop}[One-sided minimization,\text{\cite[Proposition 2.9]{hksurg}}]\label{separation}
There exists a $\overline{\delta}>0$ and $\Gamma_0<\infty$ with the following property. If $\mathcal{K}$ is an $(\alpha,\delta)$-flow ($\delta<\overline{\delta}$) in an open set $U$, with cap separation parameter $\Gamma\geq \Gamma_0$ and surgeries at scales between $\mu^{-1}s$ and $s$, and if $\overline{B} \subset U$ is a closed ball with $d(\overline{B}, \mathbb{R}^{n+1}\backslash U)\geq 20\Gamma s$, then \begin{align*}
    |\partial K_{t_1}\cap \overline{B}|\leq|\partial K'\cap \overline{B}|
\end{align*}
for every smooth comparison domain $K'$ that agrees with $K_1$ outside $\overline{B}$ and satisfies $K_{t_1}\subset K'\subset K_{t_0}$ for some $t_0<t_1$.
\end{prop}

\begin{theorem}[Global Curvature Estimate \text{\cite[Theorem 1.10]{hksurg}}]\label{hkcurv}
For all $\Lambda<\infty$, there exists $\overline\delta(\alpha)>0, \xi=\xi(\alpha,\Lambda)<\infty $ and $C_0=C_0(\alpha,\Lambda)<\infty$ with the following property. If $\mathcal{K}$ is an $(\alpha,\delta)$-flow ($\delta<\overline \delta$) in a parabolic ball $P(p,t,\xi r)$ centred at $p\in \partial\mathcal{K}_t$ with $H(p,t)\leq r^{-1}$, then 
\begin{align*}
    \sup_{P(p,t,\Lambda r)\cap \partial \mathcal{K}'} |A|\leq C_0 r^{-1}
\end{align*}
where $\mathcal{K}'$ denotes the connected component of the flow containing $p$.
\end{theorem}
\begin{remark}
 Of course, this extends to higher derivatives, $|\nabla^l A|$, as is standard for parabolic equations.
\end{remark}

\begin{defn}[$\Balpha$-controlled initial condition\text{\cite[Definition 1.15]{hksurg}}]  Let $\Balpha=(\alpha,\beta,\gamma)\in(0,n-1)\times (0,\frac1{n-1})\times(0,\infty)$. A hypersurface $M^n\subset\mathbb{R}^{n+1}$  is said to be $\Balpha$-controlled if it is
 $\alpha$-noncollapsed, $\beta$-uniformly 2-convex: $\lambda_1+\lambda_2\geq \beta H$ and $\text{max}_{x\in M}\{H(x)\}\leq \gamma$. 
\end{defn}
\begin{defn}
The surgery parameter $\mathbb{H}$ is defined as the triple \begin{align*}\mathbb{H}=\{H_\mathrm{th}, & \ H_\mathrm{neck}, H_\mathrm{trig}\}\in \mathbb{R}^3,\\ 0<H_\mathrm{th}< & \ H_\mathrm{neck}<H_\mathrm{trig}<\infty.
\end{align*} $H_\mathrm{trig}$ is the trigger curvature, once achieved the flow is stopped. $H_\mathrm{neck}$ is the mean curvature of neck points. $H_\mathrm{th}$ is the curvature that is used to determine high curvature regions of the flow.  For $\Theta<\infty$ we say $\mathbb{H}>\Theta$ if the ratios satisfy \begin{align*}
    \frac{H_\mathrm{neck}}{H_\mathrm{th}},\frac{H_\mathrm{trig}}{H_\mathrm{neck}}> \Theta
\end{align*}
We say the ratios degenerate along a sequence if these ratios tend to infinity.
\end{defn}
The definition of a mean curvature flow with surgery is made formal in the following definition.
\begin{defn}[$(\mathbb{\Balpha},\delta,\mathbb{H})$-flow \text{\cite[Definition 1.17]{hksurg}}]
Let $M^n\subset\mathbb{R}^{n+1}$ be an  $\Balpha=(\alpha,\beta,\gamma)$ controlled initial condition. An $(\Balpha,\delta, \mathbb{H})$-flow is an $(\alpha,\delta)$ flow such that:
\begin{enumerate}
    \item $H\leq H_\mathrm{trig}$ everywhere. Surgery and/or discarding occurs precisely at times $t$ when $H=H_\mathrm{trig}$ somewhere.
    \item The collection of necks in Definition \ref{alphadelta} (1) is a minimal collection of necks with curvature $H=H_\mathrm{neck}$ which separate the set $\{ H=H_\mathrm{trig}\}$ from $\{H\leq H_\mathrm{th}\}$ in the domain $K^-_t$.
    \item $K^+$ is obtained from $K^\#_t$ by discarding precisely those connected components with $H>H_\mathrm{th}$ everywhere. In particular, of each pair of facing surgery caps, precisely one is discarded.
    \item If a strong $\delta$-neck from item (2) is also a strong $\hat{\delta}$-neck for $\hat{\delta}<\delta$ then definition \ref{alphadelta} (4) also holds with $\hat{\delta}$ instead of $\delta$.
\end{enumerate}

\end{defn}

The above theory is then used to prove existence of the flow, provided one is replacing strong enough necks (controlled by $\overline{\delta}$) that are sufficiently long (controlled by $\Theta$ and the curvature estimates).
\begin{theorem}[Existence of mean curvature flow with surgery, \text{\cite[Theorem 1.21]{hksurg}}]\label{hkexist}
There are constants $\overline{\delta}=\overline{\delta}(\Balpha)>0$ and $\Theta(\delta)=\Theta(\Balpha,\delta)<\infty$ ($\delta\leq\bar{\delta}$) with the following significance.
If $\delta\leq\bar{\delta}$ and $\mathbb{H}=(H_{\textrm{trig}},H_{\textrm{neck}},H_{\textrm{th}})$ are positive numbers with
${H_{\textrm{trig}}}/{H_{\textrm{neck}}},{H_{\textrm{neck}}}/{H_{\textrm{th}}}\geq \Theta(\delta)$,
then there exists an $(\Balpha,\delta,\mathbb{H})$-flow $\{K_t\}_{t\in[0,\infty)}$ for every $\Balpha$-controlled initial condition $K_0$.

\end{theorem}

Additionally, a canonical neighbourhood theorem is proved.
\begin{theorem}[Canonical Neighbourhood Theorem, \text{\cite[Theorem 1.22]{hksurg}}]\label{hkcanon} 
For all $\varepsilon>0$, there exist $\overline{\delta}=\overline{\delta}(\Balpha)>0$, $H_{\textrm{can}}(\varepsilon)=H_{\textrm{can}}(\Balpha,\varepsilon)<\infty$ and $\Theta_\varepsilon(\delta)=\Theta_\varepsilon(\Balpha,\delta)<\infty$ ($\delta\leq\bar{\delta}$) with the following significance.
If $\delta\leq\overline{\delta}$ and $\mathcal{K}$ is an $(\Balpha,\delta,\mathbb{H})$-flow with ${H_{\textrm{trig}}}/{H_{\textrm{neck}}},{H_{\textrm{neck}}}/{H_{\textrm{th}}}\geq \Theta_\varepsilon(\delta)$,
then any $(p,t)\in\partial \mathcal{K}$ with $H(p,t)\geq H_{\textrm{can}}(\varepsilon)$ is $\varepsilon$-close to either
(a) a $\beta$-uniformly $2$-convex ancient $\alpha$-noncollapsed flow,
or (b) the evolution of a standard cap preceded by the evolution of a round cylinder.
\end{theorem}

A consequence of the canonical neighbourhood theorem is the classification of discarded components. This result allows one to use surgery to decompose the topology of the original hypersurface.
\begin{theorem}[Discarded components, \text{\cite[Corollary 1.25]{hksurg}}]\label{discarded}
For $\varepsilon>0$ small enough, for any $(\Balpha,\delta,\mathbb{H})$-flow with $H_\mathrm{neck}/H_\mathrm{th},H_\mathrm{trig}/H_\mathrm{neck}>\Theta_\varepsilon(\delta)$, and $H_\mathrm{th}>H_\mathrm{can}(\varepsilon)$, all discarded components are diffeomorphic to $\overline D^{n+1}$ or $\overline D^{n}\times \mathbb{S}^1$.
\end{theorem}

\section{Definitions for Local Surgery}

Let $\mathcal{M}$ be an $n$-dimensional unit-regular, cyclic (mod 2) integral Brakke flow that encounters only multiplicity one spherical or neck-pinch singularities, evolving from the smoothly embedded, closed hypersurface $M^n\subset\mathbb{R}^{n+1}$. We will always presume these singularities are multiplicity one. We fix a neck separation parameter $\Gamma_0$ that satisfies the conclusions of Proposition \ref{separation}, and a $\bar\delta>0$ that satisfies the conclusions of Theorem \ref{hkexist} and Theorem \ref{hkcanon}.

All of the above definitions for surgery make use of the `fattened' flow, where at each time $K_t$ is defined to be the set such that the boundary $\partial K_t = M_t$ is the motion by mean curvature from the initial hypersurface $M$. Since the flow is mean convex, the direction of flow is always into such a $K$.

With no assumption on the initial mean curvature, $\mathcal{M}$ can have `outward' necks, where the mean curvature vector (direction of flow) is pointing exterior to the compact set the hypersurface bounds. Observe, however, that the mean convex neighbourhood conjecture gives a neighbourhood of the singularity in which the mean curvature vector always points in the same direction. Recall, we are considering Brakke flows that are cyclic (mod 2), so the ambient $\mathbb{R}^{n+1}$ is  separated (at almost every time) into two components by the support of the Brakke flow. Let $\Omega$ be a set such that $\mathcal{M}\cap \Omega$ is 2-convex. Observe, this gives a `local orientation' in the following sense. We say the set $K_t$, with $\partial K_t\backslash\partial \Omega =M_t\cap \Omega$ is the local interior if $\mathbf{H}$ points into $K_t$. 

We use the same definition for the local interior of a surgery flow. Such a definition will be shown to be well defined in the definition of our flow with surgery.

\begin{defn}[Neck replacement]\label{genericnecktocap}
We localize definition \ref{necktocap} by using the above `local interior' $K_t$ as opposed to the interior of the entire flow.
\end{defn}
\begin{remark}
In this local sense, we still have the chain of inclusions 
\begin{align*}
    K^+_{t_i}\subseteq K^\#_{t_i}\subseteq K^-_{t_i}
\end{align*}
This is important for lemma \ref{barriers} in order to replicate the argument of \cite{lauer}.

Note, we will not have this sequence of inclusions for the interior of the surgery flow. Such a statement would not be true for outward necks: the caps are glued inside the solid neck, which equates to being exterior of the pre-surgery hypersurface.
\end{remark}

\begin{defn}[Locally $\alpha$-noncollapsed]\label{localandrews}
Let $M^n\subset\mathbb{R}^{n+1}$ be a smooth, closed hypersurface bounding the region $\Omega$. Suppose $M$ is mean convex in the open balls $B(\mathbf{y},2r)$. We say $M$ is \textit{locally $\alpha$-noncollapsed in $B(\mathbf{y},r)$} if
\begin{enumerate}[(a)]
\item $H(\mathbf{x})>1/r$ for $x\in M \cap B(\mathbf{y},r)$, and 
\item There is an $\alpha>0$ such that the balls $B_\mathrm{int}\subset \overline\Omega$ and $B_\mathrm{ext}\subset \mathbb{R}^{n+1}\backslash\Omega$ of radius $\alpha/H(\mathbf{x})$ situated either side of the hypersurface, with $x\in\partial B_\mathrm{int},\partial B_\mathrm{ext}$, are contained in $B(\mathbf{y},2r)$ and each ball has no intersection with $M\cap B(\mathbf{y},2r)$.
\end{enumerate}
\end{defn}

Examining the structure of the singular set of the flow $\mathcal{M}$, we can start to build the definitions for a more general surgery.
\begin{defn}
We denote the singular set of $\mathcal{M}$ as $\mathfrak{S}$.
\end{defn}

We recall the canonical neighbourhood theorem of  \cite{chh18, chhw19}.

\begin{theorem}[Canonical Neighbourhoods \text{\cite[Corollary 1.18]{chhw19}}]\label{canonical}
 Assume $X\in\mathfrak{S}$ is a neck singularity of the flow. Then for every $\delta>0$ there exists a $R(X,\delta)>0$ with the following significance. For any regular point $X'\in P(X,R)$ the flow $\mathcal{M}'=\mathcal{D}_{\lambda}(\mathcal{M}-X')$, obtained by parabolically rescaling the original flow around $X'$ by $\lambda=|\mathbf{H}(X')|$, is $\delta$-close in $C^{\lfloor 1/\delta \rfloor}$ in $B_{1/\delta}(0)\times(-1/\delta^2,0]$ to a round shrinking sphere, round shrinking cylinder, a translating bowl soliton or ancient oval.
\end{theorem}
Motivated by this theorem, we define the following open neighbourhood of the singular set of the flow $\mathcal{M}$.

 \begin{defn}[$(\alpha,\beta)$-neighbourhood] \label{abnbhd}
 We fix \begin{flalign*} 
   &(i) \ \alpha>0, \textrm{with }  \alpha<\min\{\alpha_\text{sphere},\alpha_\text{cylinder},\alpha_\text{bowl},\alpha_\text{oval}\},&\\
   & (ii) \ \beta>0, \textrm{with }  0<\beta<\text{min}\{\beta_\text{sphere},\beta_\text{cylinder},\beta_\text{bowl},\beta_\text{oval}\}, &\\
   &(iii) \ \gamma>0.
  \end{flalign*} 
  
 Here $\alpha_\text{sphere},\alpha_\text{cylinder},\alpha_\text{bowl},\alpha_\text{oval}$ and $\beta_\text{sphere},\beta_\text{cylinder},\beta_\text{bowl},\beta_\text{oval}$ are the respective optimal $\alpha>0$ and $\beta>0$ for the shrinking sphere, cylinder, translating bowl and ancient oval.

  Let $\Balpha=(\alpha, \beta, \gamma)$. Let $M^n\subset\mathbb{R}^{n+1}$ be a hypersurface with $|A|<\gamma$ and suppose $\mathcal{M}$ is a unit-regular, cyclic (mod 2) integral Brakke flow starting from $M$ then encounters only (multiplicity-one) spherical and neck-pinch singularities. We fix an additional constant $H_\mathrm{bdd}=H_\mathrm{bdd}(\Balpha)$. An $(\alpha,\beta)$-neighbourhood, $\Omega_{(\alpha,\beta)}$, is an  open space-time neighbourhood of the singular set $\mathfrak{S}$, composed of finitely many connected components, with the following properties. 
  \begin{enumerate}[$(i)$]
      \item For every regular point $X\in\mathcal{M}\cap\Omega_{(\alpha,\beta)}$, $|H(X)|> H_\mathrm{bdd}$. 
      \item If $X \in \mathcal{M}\cap\partial\Omega_i$, where $\Omega_i$ is a connected component of $\Omega_{(\alpha,\beta)}$, we require $|H(X)|=H_\mathrm{bdd}$.
      \item Furthermore, if $X \in \mathcal{M}\cap\partial\Omega_i$, then the flow is $\beta$-uniformly 2-convex in $P(X,2 \xi (|H(X)|)^{-1})$ and locally $\alpha$-noncollapsed in $P(X, \xi(|H(X)|)^{-1})$. 
      \item $\mathcal{M}$ is locally $\alpha$-noncollapsed in $\Omega_{(\alpha,\beta)}$ at regular points.
      \item $\mathcal{M}$ is $\beta$-uniformly 2-convex in $\Omega_{(\alpha,\beta)}$ at regular points.
      
  \end{enumerate}

  The value of $\xi=\xi(\alpha,\Lambda)$ is that given by the curvature estimates of Haslhofer--Kleiner, and depends on some $\Lambda$, which will be derived later.
  
  \begin{remark}
   Observe, the mean curvature is uniform across the boundary. 
     \end{remark}
  
  \begin{remark}\label{bddremark}
  The choice to have constant mean curvature along the boundary serves a practical purpose. 
  Later, we will specify surgeries in a flow approximating $\mathcal{M}$ only occur as long as said flow is a small graph over $\mathcal{M}$ in some neighbourhood of the boundary. We will show knowledge of the boundary data of $\mathcal{M}$ in the above fashion guarantees in the flows with surgery, via the maximum principle, that the hypotheses of the curvature estimates (Theorem \ref{hkcurv}) are satisfied in the interior.
 To be explicit, at interior points $X$, the flow in $P(X,\xi(|H(X)|)^{-1})$ will be an $(\alpha,\delta)$-flow in the sense of \cite{hksurg}.
  \end{remark}

  \end{defn}

  \begin{lemma}Let $\mathcal{M}$ be a Brakke flow with only spherical and neck-pinch singularities. For every $\Balpha$ as in Definition \ref{abnbhd}, there is a  $H_0(\Balpha, \mathcal{M})<\infty$ such that for all $H_\mathrm{bdd}>H_0$ an $(\alpha,\beta)$-neighbourhood exists.
  \end{lemma}
  \begin{proof}
  Fix $\Balpha$ satisfying the assumptions of Definition \ref{abnbhd}, and take $\varepsilon<(2\xi)^{-1}$. Additionally, we take $\varepsilon$ small enough that if a flow is $\varepsilon$-close an ancient, asymptotically cylindrical flow, then it is $\beta$-uniformly 2-convex.
  
  By the canonical neighbourhood theorem, Theorem \ref{canonical}, and the compactness of the singular set, there is an $r>0$ such that any regular point in the parabolic cylinder $P(Y,r)$, centred at $Y\in\mathfrak{S}$ is $\varepsilon$-close to one of the ancient, asymptotically cylindrical flows (at scale of the mean curvature).
  
  This radius can be taken such that at any interior regular point the flow is locally $\alpha$-noncollapsed.
  
  The union of the above cylinders, $\cup_{Y\in \mathfrak{S}} P(Y,r) $, defines a cover of the singular set. Observe, in each connected component the mean curvature has a single sign (a local orientation). Let $\{X_i\}_{i\in\mathbb{N}}$ be a sequence of regular points contained in a single connected component that accumulate in $\mathfrak{S}$. It is immediate from the canonical neighbourhood theorem that $H(X_i)\to \infty$.
  
  Hence, we can fix a $H_\mathrm{bdd}$ sufficiently large that $\Omega:=\{X\in\mathrm{reg}(\mathcal{M})\  \mathrm{s.t.} \ |H(X)|>H_\mathrm{bdd}\}\Subset \cup_{Y\in \mathfrak{S}} P(Y,r) $.

  Observe, $\mathrm{reg}(\mathcal{M})$ is relatively open in $\mathrm{supp}(\mathcal{M})$, so $\Omega$ is a relatively open set in $\mathrm{supp}(\mathcal{M})$. Moreover, the mean convex neighbourhood theorem shows that we can include singular points, provided they are spherical or neck-pinch singularities, i.e.~ $\Omega'=\{ X\in \mathrm{reg}(\mathcal{M}) \ | \ |H(X)|>H_{\mathrm{bdd}}\}\cup \mathfrak{S}$ is open in $\mathrm{supp}(\mathcal{M})$. The topology of $\mathrm{supp}(\mathcal{M})$ is inherited from the standard parabolic topology of space-time, $\mathbb{R}^{n+1,1}$. Thus, there is an open set $U$ in $\mathbb{R}^{n+1,1}$ such that $U\cap \mathrm{supp}(\mathcal{M})= \Omega'$. $\Omega_{(\alpha,\beta)}$ can be taken as any collection of such open sets in space-time. Hence, $\Omega_{(\alpha,\beta)}$ is an open space-time neighbourhood of the singular set. We can assume this neighbourhood has finitely many connected components since the singular set is compact.
  
  Finally, the $\beta$-uniform 2-convexity and $\alpha$-noncollapsedness for $X\in\mathcal{M}\cap\partial\Omega_i$ is immediate from the choice of $\varepsilon$ in the canonical neighbourhood theorem.
  \end{proof}

\begin{defn}[Neighbourhood of the boundary]\label{bddnbhd}
For a connected component $\Omega_i$, we define
\begin{align*}
    N_i = \bigcup_{X\in \partial \Omega_i} P(X,2 \xi H_\mathrm{bdd}^{-1}).
\end{align*}
Where $P(X,2 \xi H_\mathrm{bdd}^{-1})$ is the backwards parabolic cylinder centered at $X$. Observe, as specified in Definition \ref{abnbhd}, $\mathcal{M}\cap P(X,2 \xi H_\mathrm{bdd}^{-1})$ will be smooth and $\beta$-uniformly 2-convex.
\end{defn}

We now define a flow similar to the mean convex $(\alpha,\delta)$-flows of \cite{hksurg}. It is a unit-regular cyclic mod 2 Brakke flow with the replacement of (smooth) $\delta$-necks by caps.

\begin{defn}[$(\alpha,\delta)$-Brakke flow]\label{brokenflow}
Compare definition \ref{alphadelta}.

Let $M^{n}\subset \mathbb{R}^{n+1}$ be a compact, smoothly embedded hypersurface.  Let  $\mathcal{M}$ a unit-regular, cyclic (mod 2) Brakke flow emerging from $M$ that encounters only (multiplicity one) spherical and neck-pinch singularities.

An $(\alpha, \delta)$-Brakke flow is defined as the collection of unit-regular cyclic (mod 2) Brakke flows
\begin{align*}
    \{\mathcal{M}^i\}=\{\mu^i_t\}_{t\in [t_{i-1},t_i]}, (i=1,\ldots,k+1;\  0=t_0<\cdots<t_k<t_{k+1}=t_\mathrm{Ext}),
\end{align*}

with the following properties. We adopt the standard notation of `calligraphic' $\mathcal{M}$ to denote flows, and `roman' $M_t$ the $t$-time slice of $\mathcal{M}$. Superscripts will remain consistent between flows and timeslices in flows with surgery.
\begin{enumerate}[(i)]
    \item $\mathcal{M}^i$ is a smooth flow for $1\leq i\leq k$. That is, surgery is only performed if the flow is smooth.
    \item For each $i=1,\ldots k$, we identify in $M^i_{t_i}=:M^-_{t_i}$, the final time slice of the smooth mean curvature flow $\mathcal{M}^i$, a collection of disjoint strong $\delta$-necks contained in $\Omega_{(\alpha,\beta)}$. Each neck is replaced, provided the next point is satisfied, by pairs of standard caps as in Definition \ref{necktocap}, creating the possibly disconnected hypersurface $M^\#_{t_i}$.
    \item Necks at time $t_i\in\{t_1,\ldots,t_k\}$ contained in $\Omega_j$, a connected component of $\Omega_{(\alpha,\beta)}$, are only replaced by caps if the flow $\mathcal{M}^i$ can be written as a $\delta$-graph over $\mathcal{M}$ in the boundary neighbourhoods $N_j$ at time $t_i$. This is to ensure that the curvature estimate of \cite{hksurg} carries over to the surgery flow. See Remarks \ref{item3} and \ref{reasonforcontrol}. If this condition fails, we treat the last time surgeries were successfully performed as $t_k$ and we continue as in item (vi). Note, we allow the case where being a graph over the boundary at time $t_i$ is `vacuously true' i.e. $\mathcal{M}^i\cap\partial\Omega_j=\emptyset,\mathcal{M}^i\cap\Omega_j\neq\emptyset$. Indeed, if a component of the flow is contained entirely in $\Omega_{(\alpha,\beta)}$, then it satisfies the assumptions of $\alpha$ non-collapsedness and $\beta$-uniform 2-convexity by the maximum principle. 
    \item The initial timeslice of $\mathcal{M}^{i+1}$, $M^{i+1}_{t_i}:=M^+_{t_i}$ is obtained from the post-surgery hypersurface $M^{\#}_{t_i}$ by dropping some connected components contained in $\Omega_{(\alpha,\beta)}$.
    \item There exists $s_\#>0$ which depends only on the Brakke flow $\mathcal{M}$ such that all necks in item (i) have radius $s\in [\mu^{-1/2}s_\#,\mu^{1/2}s_\#]$\footnote{ $\mu \in[1,\infty)$ is a constant that quantifies the notation of surgeries at comparable scales. See \cite[Convention 1.2]{hksurg}}.
    \item We allow the flow $\mathcal{M}^{k+1}$ to develop as a unit-regular Brakke flow until its extinction at time $t_{k+1}=t_\mathrm{Ext}$. Specifically, we choose the integral, unit-regular, cyclic (mod 2) Brakke flow whose support is the outer flow from the initial condition of $\mathcal{M}^k$. See Hershkovits--White, \cite{hershwhite}, where such a flow is constructed.

\end{enumerate}
\end{defn}

\begin{remark}
 In item (i), we require that $M^i_{t_i}$ is a smooth hypersurface for neck replacement to occur. Thus, after neck replacement the flow can be continued as an integral, unit-regular, cyclic (mod 2) Brakke flow by elliptic regularisation. It should be possible to weaken this requirement to being an integral current, however, this is not needed for the purposes of the current work. The choice of outer flow is important later, for understanding barriers to flows with surgical modification.
\end{remark}

\begin{remark}\label{item3}
 Item (iii) requires the $(\alpha,\delta)$-Brakke flow can be written as a $\delta$-graph over $\mathcal{M}$ in $N_i$. By this we mean, the surgery flow is $\delta$-close to $\mathcal{M}$ in $C^{\lfloor\frac1{\delta} \rfloor}(N_i)$. Whilst imposing such a condition may seem unmotivated, it occurs naturally when considering sequences of smooth flows that converge to a smooth limit. We discuss how our flows with surgery converge in Section \ref{barriersandstability}. 
\end{remark}
\begin{remark}\label{reasonforcontrol}We use the $\delta$-graphical condition to ensure that along the boundary of $\Omega_{(\alpha,\beta)}$, the surgery flow is $\beta$-uniformly 2-convex and $\alpha$-noncollapsed, provided $\delta>0$ is taken sufficiently small. The size of the required $\delta$ will depend on $H_\mathrm{bdd}$ and, of course, our choice of $\alpha$ and $\beta$. We can then promote this to interior control by the maximum principle.
Demanding control in a neighbourhood of the boundary (as opposed to just on the boundary) addresses two problems. Firstly, we need to use a two point maximum principle to show interior $\alpha$-noncollapsedness as in\cite{andrews}. We discuss why this graphical condition in the boundary provides sufficient control of the geometry of the flows with surgery to apply a two point maximum principle in Remark \ref{andrewsdiscussion}. Secondly, by enforcing a boundary graphical condition in the definition of the $(\alpha,\delta)$-Brakke flows, we ensure the hypotheses of the Haslhofer--Kleiner curvature estimate are satisfied at all interior points, before the final time of surgery. This follows essentially from the triangle inequality and the maximum principle. For details, see Theorem \ref{bddctrl}.
\end{remark}
\begin{remark}
 It is important to stress that the uniform backward control of 2-convexity and noncollapsedness along the boundary is fundamental in being able to apply the curvature estimate for our choice of $\Lambda$. Note, this control is not needed if the mean curvature tends to infinity, only when one expects the curvature to remain bounded. For example, this argument is not needed when applying the curvature estimates in the Canonical Neighbourhood Theorem of Haslhofer--Kleiner, but is needed for showing surgery accumulates in the singular set.
\end{remark}
\begin{remark}
In the formalism of Haslhofer--Kleiner surgery, $\alpha$ and $\beta$ are controlled by the initial condition. In this flow, these parameters are controlled locally from the values on the boundary by the maximum principle. 
\end{remark}

We now define the weak surgical flows.  The key deviations are that (a) the flow can become singular, and (b) the requirement that surgery only takes place in a predetermined neighbourhood of the singular set of the flow $\mathcal{M}$. Whilst this initially may feel restrictive, it is entirely natural. See Section \ref{barriersandstability}.

\begin{defn}[Weak $(\Balpha, \delta, \mathbb{H})$-flow]\label{weakflow}
Let $M^{n}\subset \mathbb{R}^{n+1}$ be a compact, smoothly embedded hypersurface be a $\gamma$-controlled initial condition. Let $\mathcal{M}$ be a unit-regular, cyclic (mod 2) Brakke flow emerging from $M$ that encounters only (multiplicity one) spherical and neck-pinch singularities. For a fixed $\Balpha$ (as above), $\delta>0$ and surgery parameters $\mathbb{H}$ we define $\mathcal{M}_\mathbb{H}$ as the \textit{weak $(\Balpha, \delta, \mathbb{H})$-flow} or \textit{weak surgery flow derived from $\mathcal{M}$} as the $(\alpha, \delta)$-Brakke flow that satisfies the following conditions:
\begin{enumerate}[(i)]
    \item All surgeries take place inside the $(\alpha,\beta)$-neighbourhood of the singular set of $\mathcal{M}$, the region where the original flow is $\alpha$-noncollapsed and $\beta$-uniformly 2-convex.
    \item Surgeries and/or discarding takes place at times $t$ when $|\mathbf{H}|=H_\mathrm{trig}$ somewhere in $\Omega_{(\alpha,\beta)}$. Note, we actually allow $|\mathbf{H}|$ to exceed $H_\mathrm{trig}$ in the flow outside the region where we perform surgery.
    \item The collection of necks is minimal, and the necks are of curvature $|H_\mathrm{neck}|$. The necks separate the set $\{|\mathbf{H}|=H_\mathrm{trig}\}$ from $\{|\mathbf{H}|\leq H_\mathrm{th}\}$.
    \item The smooth hypersurface $M^+_t$ is obtained from $M^{-}_t$ by dropping some smooth components of mean curvature $|\mathbf{H}|>H_\mathrm{th}$ contained in $\Omega_{(\alpha,\beta)}$. In particular, for each pair of facing surgery caps, precisely one is discarded.
    \item If a strong $\delta$-neck is also a strong $\hat\delta$ neck for $\hat\delta<\delta$ then item (iv) of definition \ref{brokenflow} holds with $\hat\delta$ instead of $\delta$.
\end{enumerate}

\end{defn}

\begin{remark}
Item (v) is the stipulation that if a $\delta$-neck sits inside a stronger $\hat \delta$-neck, then the surgery is performed in a `better' way, that is closer to the ideal cylinder and cap. This is an essential component of self-improvement.
\end{remark}

\begin{remark}
We allow the flow to continue as a unit-regular Brakke flow if a (possibly non-generic) singularity forms after the last surgery. Note that we cannot be certain such a continuation is unique. We gain control of the singular behaviour via the barriers constructed in Section 4, in particular showing that any singularities will be spherical or neck-pinch singularities (and thus the continuation is well defined). In Section 5, we will show that giving control back to $H_\mathrm{trig}$ gives a smooth surgery in the same sense as \cite{hksurg}. 
\end{remark}

Consider the following examples of weak surgery flows.
\begin{example}
The shrinking sphere is a weak $(\Balpha,\delta,\mathbb{H})$-surgery flow for all values of $\mathbb{H}$, if one chooses not to drop components of high curvature.
\end{example}
\begin{example}
Fix $\mathbb{H}$. The shrinking sphere that vanishes once the mean curvature reaches $H_\mathrm{th}$ is a weak $(\Balpha,\delta,\mathbb{H})$-surgery flow.
\end{example}
\begin{example}
 Fix $\Balpha$ and $\delta>0$. Let $M$ be an $\Balpha$-controlled initial condition. Then, there is a $\mathbb{H}$ given by \cite{hksurg} such that the $(\Balpha, \delta,\mathbb{H})$ mean curvature flow with surgery of \cite{hksurg} exists. It is a weak $(\Balpha,\delta,\mathbb{H})$-surgery flow.
\end{example}
\section{Barriers and Stability}\label{barriersandstability}

We now develop the tools for controlling the weak surgery flows. In the first half of this section, we show that the unit-regular Brakke flow from hypersurfaces equidistant to the initial hypersurface act as barriers to our weak surgery flows, provided the surgery scale is large enough. The existence of these barriers requires the recent technical result of \cite{ccms20}, concerning the connectedness of the singular set for flows with singular set of small Hausdorff dimension. Indeed, such a result is critical as one needs a way to show higher multiplicities cannot develop.
We then tackle the problem of stability of the surgery flows. The parabolic nature of mean curvature flow means that changing the flow in one location can affect other regions at infinite speed. Whilst this problem cannot be completely avoided, showing the surgery parameters can be chosen such that surgeries change the flow in a manner that is `stable’ with respect to the unmodified flow is sufficient. Recalling the definition of the $(\alpha,\beta)$-neighbourhood, one can see that if we can show suitable control in $N_i$, a neighbourhood of the boundary of a connected component of the $(\alpha,\beta)$-neighbourhood, then in the interior our flow with surgery will locally look like a $(\Balpha,\delta,\mathbb{H})$-flow of Haslhofer—Kleiner. In Section 5, this is precisely how we will show that their theory can be applied directly to deduce existence of a smooth flow with surgery. Said boundary control is achieved by a local convergence result. In showing this, we additionally prove the stronger result that the weak flows with surgery converge to the unmodified flow as Brakke flows away from the singular set.

 For the following, we will suppose that $M^n\subset\mathbb{R}^{n+1}$ is a closed, smoothly embedded hypersurface and that there is a unit-regular, cyclic (mod 2) Brakke flow $\mathcal{M}$ emerging from $M$ that encounters only multiplicity one spherical and neck-pinch singularities. A priori, such a flow is not unique, however, combining recent results we get the following uniqueness result.

\begin{theorem}\label{bigunique}
 Let $M^n\subset\mathbb{R}^{n+1}$ be a closed, smoothly embedded hypersurface. If there is a unit-regular cyclic (mod 2) Brakke flow $\mathcal{M}$ emerging from $M$ that encounters only multiplicity one spherical and neck-pinch singularities, then the level set flow does not fatten. In particular, $\mathcal{M}$ is unique.
\end{theorem}
\begin{proof}
Recall that the support of $\mathcal{M}$ defines a weak set flow, and thus is contained in the level set flow of $M$.
Let $\mathcal{N}$ be the unit-regular Brakke flow whose support is the outer flow $\{M_t\}$. The existence of such a flow is proven in \cite{hershwhite}. The uniqueness of smooth mean curvature flow implies that $\mathcal{M}$ and $\mathcal{N}$ agree up to the first singular time. Thus, their supports agree at the first singular time. Since $\mathcal{M}$ has only spherical and neck-pinch singularities, the flow $\mathcal{N}$ cannot fatten at the first singular time, $t_0$, \cite{hershwhite}. Moreover, stratification, \cite{bw97}, yields that the singular set of $\mathcal{M}$ has parabolic Hausdorff dimension at most one. Hence, by Theorem \ref{smallsing}, (\cite[Theorem F.4]{ccms20}), the regular sets of $\mathcal{M}$ and $\mathcal{N}$ are connected, and thus we have unit density at smooth points. Thus, the flows agree as Brakke flows up to the first singular time. This argument can be iterated since the flow is compact. i.e.~ For the two flows to differ, the outer flow must encounter a non-spherical or non-neck-pinch singularity, which cannot happen as the flows agree back in time. Thus, $\mathcal{M}=\mathcal{N}$. In particular, the outer flow has only spherical and neck-pinch singularities and hence does not fatten, \cite[Theorem 1.19]{chhw19}. 

Since the support of any Brakke flow defines a weak set flow, the non-fattening and connectedness of the regular set show that $\mathcal{M}$ is the unique unit-regular flow.

\end{proof}

Thus, it is sufficient to suppose $\mathcal{M}$ has only spherical and neck-pinch singularities.
 
We also pick a $\varepsilon_0= \varepsilon_0(M)>0$ sufficiently small, such that for $-\varepsilon_0\leq \varepsilon\leq \varepsilon_0$ the hypersurfaces $M_\varepsilon = \{\text{dist}(\cdot, M) = \varepsilon\}$, where $\text{dist}(\cdot, M)$ is the signed distance function to $M$, are smooth.

\begin{lemma}\label{convergence}
Let $\varepsilon<\varepsilon_0$, and let $\mathcal{M}_{\pm \varepsilon}$ be unit-regular cyclic (mod 2) Brakke flows emerging from the hypersurfaces $M_{\pm\varepsilon}$. Then, \begin{align*} \lim_{\varepsilon\to 0} \mathcal{M}_{\pm\varepsilon} = \mathcal{M} \end{align*}
as Brakke flows.
\end{lemma}
\begin{proof}
 We prove the statement for the $+\varepsilon$ flows, as the proof for the $-\varepsilon$ flows will be identical.
 
Smooth convergence of $ \mathcal{M}_{\varepsilon} \to\mathcal{M}$ holds up to the first singular time of $\mathcal{M}$. For later times we consider the following.

Let $\{\varepsilon_i\}_{i\in\mathbb{N}}$ be a positive null sequence, and consider the flows $\mathcal{M}_{\varepsilon_i}$. By the convergence result of Ilmanen \cite{ilmanen}, there is a unit-regular flow $\tilde{\mathcal{M}}$ such that $\mathcal{M}_{\varepsilon_i} \rightharpoonup \tilde{\mathcal{M}}$. In particular, since the level set flow from $M$ does not fatten, we have $\text{supp}(\tilde{\mathcal{M}}) \subseteq \text{supp}(\mathcal{M})$.

We now proceed via the logic of Theorem \ref{smallsing} \cite[Appendix F]{ccms20}.

Since $\mathcal{M}$ has only spherical and neck-pinch singularities, stratification, \cite{bw97}, yields that the singular set has parabolic Hausdorff dimension at most one, so by Theorem \ref{smallsing} $\mathcal{M}$ has connected regular set. Indeed, by considering paths that connect to the initial time avoiding the singular set and noting that $\tilde{\mathcal{M}}$ is unit regular, we see that the density of $\tilde{\mathcal{M}}$ is equal to that of $\mathcal{M}$ at all regular points. Since the singular set of $\mathcal{M}$ has small measure, we have $\tilde{\mathcal{M}}= \mathcal{M}$.

This is true for all null sequences $\{\varepsilon_i\}$, hence the above argument shows $\mathcal{M}_{+\varepsilon}$ converges to $\mathcal{M}$.
\end{proof}
\begin{remark}
Note, for small $\varepsilon>0$ the barrier flows have only spherical and neck--pinch singularities.  This follows from the resolution of the mean convex neighbourhood conjecture, \cite{chh18, chhw19} and the extension to near-by flows by Schulze--Sesum \cite{schulzesesum}.
\end{remark}

\begin{lemma}\label{distances}Let $M,M_{\pm\varepsilon}$ be as above. Then, for every $t$ where both flows are defined,
$|d(M_t,{M}_{{\pm\varepsilon},t})|\geq \varepsilon$.
\end{lemma}
\begin{proof}
 Follows from the standard avoidance principle for Brakke flows, see \cite{ilmanen}.
\end{proof}

\begin{defn}
 We will call the unit-regular Brakke flows $\mathcal{M}_{\pm\varepsilon}$ the \textit{$\varepsilon$-barriers}.
 
 We take the convention that $M_{+\varepsilon}$ is the hypersurface in the interior of $M$. $M_{-\varepsilon}$ is thus in the exterior.
 \end{defn}

 \begin{lemma}\label{barriers}($\mathcal{M}_{\pm\varepsilon}$ as Surgical Barriers) Let $M$ be as above.
 Fix $\varepsilon$, with $ 0<\varepsilon<\mu(M)$. Then, there exists a $H(\varepsilon)<\infty$ such that any weak $(\Balpha, \delta,\mathbb{H})$ surgical flow with $H_\mathrm{th}>H(\varepsilon)$ avoids $\mathcal{M}_{\pm\varepsilon}$. In particular, the distance between the barriers and surgery flow is non-decreasing.
 \end{lemma}

 \begin{proof}

 It is well known that the distance between two non-intersecting Brakke flows is non-decreasing, (avoidance principle \cite{ilmanen}). Provided the distance is not decreased across surgery, the claim follows.
 
 We hence check the behaviour at time of surgery. Without loss of generality, we consider only one of the barriers at inward and outward necks. The proof for the other barrier will follow identically.
 
 Let $\mathcal{M}_{+\varepsilon}$ be the evolution of the hypersurface in the interior of $M$. We follow the argument as outlined in \cite{hksurg}.
 \begin{claim} Let $t$ be a surgery time at an inward neck for the surgical flow $\mathcal{M}_\mathbb{H}$.
 For every $r>0$, there is a $H_\mathrm{min}(r)<\infty$ such that if $H_\mathrm{th}>H_\mathrm{min}$ and $B(\textbf{p},r)\subset \textrm{int}(M_{\mathbb{H},t^-})$, then $B(\textbf{p},r)\subset \textrm{int}(M_{\mathbb{H},t^+})$. 
 
 \end{claim}
\begin{proof}

Fix $r>0$.
 There are two regions one needs to check:
\begin{enumerate}
    \item The collection of necks. For each neck we consider its interior $K$ (See Definition \ref{genericnecktocap}). Following the argument of \cite[Theorem 1.25]{hksurg}, for sufficiently large $H_\mathrm{th}$, a ball of radius $r$ cannot be contained in $K$, as it will be a long and thin neck.

 \item The dropped components. If the ball were contained in the interior of a discarded component, then the discarded component would have a point with $|H|\leq nr^{-1}$. Discarded components have $|H|\geq H_\mathrm{th}$, thus picking $H_\mathrm{th}>nr^{-1}$ is sufficient to prove the claim.
 \end{enumerate}
 \end{proof}

  \begin{claim} Let $t$ be a surgery time at an outward neck for the surgical flow $\mathcal{M}_\mathbb{H}$.
 For every $r>0$, there is a $H_\mathrm{min}(r)<\infty$ such that if $H_\mathrm{th}>H_\mathrm{min}$ and $B(\textbf{p},r)\subset \textrm{int}(M_{\mathbb{H},t^-})$, then $B(\textbf{p},r)\subset \textrm{int}(M_{\mathbb{H},t^+})$. 
 \end{claim}
 \begin{proof}
 Recall, at outward necks, the `interior' of the neck is exterior to the flow. The caps are glued inside the cylinder. Thus, if $B(\textbf{p},r)\subset \textrm{int}(M_{\mathbb{H},t^-})$, then we have $B(\textbf{p},r)\subset \textrm{int}(M_{\mathbb{H},t^+})$ for all values of $H_{th}$.
 \end{proof}
 For the other barrier, we consider $B(\textbf{p},r)\subset \textrm{ext}(M_{\mathbb{H},t^-})$. The proofs are identical, but for the oppositely oriented necks.

  To illustrate how the above claims prove the distance is non-decreasing, consider the following. Fix $\varepsilon>0$ and choose the surgery parameter $\mathbb{H}$ such that $\hthic>H_\mathrm{min}(\varepsilon)$. Let $t$ be the first time of surgery. We now consider the balls $B(\mathbf{x},d(\mathbf{x},M_{\mathbb{H},t^-}))$, where $d(\cdot,M_{\mathbb{H},t^-})$ is the distance of a point to the hypersurface $M_{\mathbb{H},t^-}$, for each point $\mathbf{x}$ in the $t$ timeslice of $\mathcal{M}_{\pm\varepsilon}$. Clearly any such ball will lie entirely on one side of $M_{\mathbb{H},t^-}$. Since flows with surgery are simply smooth flows up to time $t$, the avoidance principle shows that the radius, $r=r(\mathbf{x})$, of each ball must have $r\geq\varepsilon$. We deduce from the above claims that each of the discussed balls in the interior (resp. exterior) of $M_{\mathbb{H},t^-}$ will be interior (resp. exterior) to $M_{\mathbb{H},t^+}$ after surgery, as $\hthic>H_\mathrm{min}$. Thus, the distance of $M_{\mathbb{H},t^+}$ to either barrier at time $t$ cannot be less than that of $M_{\mathbb{H},t^-}$. Since a surgical flow is a Brakke flow between surgery times, the avoidance principle allows for the argument to be repeated at all later surgery times. We conclude the distance between the barriers and the surgical flow is non-decreasing along the entire flow.
 
\end{proof}
\begin{remark}
 Interior and exterior are well defined because we are considering smooth hypersurfaces at times of surgery. Note, the property of `separating' the inner and outer barriers is preserved through surgery, in the sense that at any time, any path connecting the inner and outer barriers must pass through the flow with surgery. In addition, such a separation property is valid for all times after the last surgery by our choice to continue the surgery flow as the unit-regular cyclic (mod 2) Brakke flow whose support is the outer flow.
\end{remark}
\begin{corollary}[Hausdorff Convergence]\label{Hausdorff}
Taking the limit as $\hthic\to \infty$, the weak flows with surgery from $M$ converge to the level set flow from $M$ in the Hausdorff sense.
\end{corollary}

\begin{proof}
Recall, we use the convention that $M_{+\varepsilon}$ is interior to $M$.
Let $U$ be the compact set bounded by $M$, and $U'=\overline{U^c}$. Similarly, denote $U_{\pm\varepsilon}$ as the compact sets with $\partial U_{\pm\varepsilon}=M_{\pm\varepsilon}$, and $U'_{\pm\varepsilon}=\overline{U_{\pm\varepsilon}^c}$.
It is clear that for all $\varepsilon_1>\varepsilon_2>0$ we have
\begin{align*}
    U_{-\varepsilon_1}\supset U_{-\varepsilon_2}\supset \ &U \supset U_{+\varepsilon_2}  \supset U_{+\varepsilon_1}\\
    U'_{+\varepsilon_1}\supset U'_{+\varepsilon_2}\supset \ &U'\supset U'_{-\varepsilon_2} \supset U'_{-\varepsilon_2}
\end{align*}
Using the notation of \cite{hershwhite}, we denote the space-time track of the level set flow from $U,U'$ as $\mathcal{U},\mathcal{U}'$. We have 
\begin{align*}
    \mathcal{U}_{-\varepsilon_1}\supset \mathcal{U}_{-\varepsilon_2}\supset\ &\mathcal{U} \supset \mathcal{U}_{+\varepsilon_2} \supset \mathcal{U}_{+\varepsilon_1}\\
     \mathcal{U}'_{+\varepsilon_1}\supset \mathcal{U}'_{+\varepsilon_2}\supset\ &\mathcal{U}' \supset \mathcal{U}'_{-\varepsilon_2}\supset \mathcal{U}'_{-\varepsilon}
\end{align*}
By Lemma \ref{convergence}, we can take $\varepsilon>0$ small enough such that $\mathcal{M}_{\pm \varepsilon}$ has only spherical and neck-pinch singularities. Thus, the level set flow from $M_{\pm\varepsilon}$ does not fatten, and hence $\partial \mathcal{U}_{+\varepsilon}=\partial \mathcal{U}'_{+\varepsilon}=\mathrm{supp}(\mathcal{M}_{+\varepsilon})$. 

We define the closed sets $\mathcal{K}_\varepsilon:=\mathcal{U}'_{+\varepsilon}\cap \mathcal{U}_{-\varepsilon}$ and $K(t):=\{x\in\mathbb{R}^{n+1} \ | \ (x,t)\in \mathcal{K}\}$.  Note, the space-time boundary of $\mathcal{K}_\varepsilon$ is $\partial \mathcal{K}_{\varepsilon}=\mathrm{supp}(\mathcal{M_{+\varepsilon}})\sqcup\mathrm{supp}(\mathcal{M_{-\varepsilon}})$. Recall, these flows are disjoint by the avoidance principle.

By Lemma \ref{barriers}, for every $\varepsilon>0$, we can find a $H(\varepsilon)<\infty$ such that any weak surgery flow $\surg$ with $\hthic>H$ avoids $\mathcal{M}_{\pm \varepsilon}$. Indeed, we see that $\surg\subset \mathcal{K}_\varepsilon$ and at every time $t\geq0$ where both $\mathcal{M}_{\pm\varepsilon}$ are non-empty, $\surg$ `separates', in the sense that any (space-like) curve joining $M_{+\varepsilon}(t)$ to $M_{-\varepsilon}(t)$  must pass through $M_{\mathbb{H},t}$. The corollary will follow immediately from the following claim.

\begin{claim}
 $\mathcal{K}_\varepsilon$ converges to $\mathrm{supp}(\mathcal{M})=\{(x,t)\in \mathbb{R}^{n+1}\times\mathbb{R} \ \mathrm{s.t.} \ x \in F_t(M)\}$ in the Hausdorff sense as $\varepsilon\to 0$.
\end{claim}
\begin{proof}
By construction, $\mathrm{supp}(\mathcal{M})\subset \mathcal{K}_\varepsilon$ for all $\varepsilon>0$, i.e. for all $\xi>0$, $\mathrm{supp}(\mathcal{M})$ is always in the $\xi$ neighbourhood of $\mathcal{K}_\varepsilon$.

Observe, for $\varepsilon_1>\varepsilon_2>0$, we have $\mathcal{K}_{\varepsilon_2} \subset\mathcal{K}_{\varepsilon_1}$. Thus, it is sufficient to show $\mathrm{supp}(\mathcal{M})\supseteq\cap_{\varepsilon\to 0}\mathcal{K}_\varepsilon$. (Clearly the reverse inclusion is true). We do this by showing $\mathcal{K}:=\cap_{\varepsilon\to 0}\mathcal{K}_\varepsilon$ defines a weak set flow from $M$.

Observe, at $t=0$, we have $\cap_{\varepsilon\to 0} K_\varepsilon(0)=M$, as $K_\varepsilon(0)=\{x\in\mathbb{R}^{n+1}\  | \ d(x,M)\leq\varepsilon \}$ and $M$ is closed. 

Given any smooth compact hypersurface $N$ that is disjoint from $M$, we can find an $\varepsilon>0$ such that $K_\varepsilon(0)\cap N = \emptyset$, simply by taking $\varepsilon\leq d(M,N)$. It is immediate from the definition of $\mathcal{K}_\varepsilon$ that it will be disjoint from the space-time track of the mean curvature flow from $N$. Indeed, $\mathcal{K}$ must avoid every smooth mean curvature flow that is initially disjoint with $M$. Thus, $\mathcal{K}$ defines a weak set flow from $M$. Since $\mathrm{supp}(\mathcal{M})$ is the space-time track of the level set flow, it must contain $\mathcal{K}$. This follows from the definition of the level set flow as the maximal weak set flow, see \cite{ilmanen}.

\end{proof}

Indeed, we have shown that the `gap' between $\mathcal{M}_{\pm\varepsilon}$, $\mathcal{K}_\varepsilon$, Hausdorff converges to $\mathrm{supp}(\mathcal{M})$ as $\varepsilon\to 0$. Since $\mathcal{M}_{\pm\varepsilon}$, the space-time boundary components of $\mathcal{K}_\varepsilon$, converge in the Brakke sense to $\mathcal{M}$, and any surgery flow with $\hthic>H(\varepsilon)$ will separate $\mathcal{M}_{\pm\varepsilon}$, we deduce $\lim_{\hthic\to \infty}\surg=\mathrm{supp}(\mathcal{M})$.

\end{proof}

Having shown Hausdorff convergence, our goal now is to establish graphical control of the weak surgery flows in the boundary neighbourhood of the $(\alpha,\beta)$ neighbourhood. This is achieved by establishing Brakke convergence in this region. We will actually show Brakke convergence on the full regular set. Consider for a moment a sequence of Brakke flows that converge in a Hausdorff sense to another Brakke flow. Improving the convergence to Brakke convergence is straight forward provided one can find a way to control multiplicity. See the proof of Proposition \ref{convergence2}, claim \ref{supportequal} onwards. Recalling the definition of an $(\alpha,\delta)$ Brakke-flow, Definition \ref{brokenflow}, inside any open space-time set that does not contain a surgery, an $(\alpha,\delta)$-Brakke flow is a unit-regular, cyclic (mod 2) Brakke flow. Thus, Brakke convergence will follow from understanding where, in a limiting sense, surgeries occur in our weak surgery flows. Indeed, we will show that the surgeries accumulate in the singular set of $\mathcal{M}$. Using what has been shown so far we can develop some intuition as why this is expected behaviour. 

Let $M^n\subset\mathbb{R}^{n+1}$ and $\mathcal{M}$ be as stated at the start of the current section. For the sake of simplicity, suppose further $\mathcal{M}$ encounters an isolated, non-degenerate neck-pinch singularity at the first singular time. Let $\surgi$ be a sequence of weak flows with surgery starting from $M$, with $\hthic^i \to \infty$. At the first time of surgery in the flow $\surgi$, we can identify a $\delta$-neck with centre $P_i$ and mean curvature $ H_{\surgi} (P_i)=\hneck^i$ that is about to under-go surgery. The sequence $\{P_i\}^\infty_{i=1}$ can be treated as a sequence of points in $\mathcal{M}$ since, by definition, the weak flows with surgery must agree with $\mathcal{M}$ up to their respective first surgery time. Since $\hneck^i\to\infty$, it is clear that the points $P_i$ must accumulate in the singular set at the first singular time. Whilst this argument works at the first time of surgery, it unfortunately cannot be applied at later surgery times, however, we can use the barriers begin to understand what is happening.  In the following we develop a general intuition, though it may be informative for the reader to keep in mind the specific example of the classic 2-convex dumbbell as the initial condition $M$ and $\mathcal{M}$ the outer flow from $M$.
First, we note $\varepsilon>0$ can be chosen small enough such that the barrier flows $\mathcal{M}_{\pm\varepsilon}$ also satisfy the canonical neighbourhood condition in our $(\alpha,\beta)$ neighbourhood. We may assume the barriers are moving monotonically towards their (global) interior inside $\Omega_{(\alpha,\beta)}$; in connected components of $\Omega_{(\alpha,\beta)}$ where flows are moving monotonically towards their exterior, simply exchange the roles of the inner and outer barriers. Secondly, we note that any weak flow with surgery (with sufficiently large $\hthic$) separates $\mathcal{M}_{\pm\varepsilon}$. Indeed, we have the set of inclusions outlined in Corollary \ref{Hausdorff}. Thus, by our avoidance principle, Lemma \ref{barriers}, surgeries can only occur in regions where the inner barrier is not present. Conversely, we see the outer barrier $\mathcal{M}_{-\varepsilon}$ can only pinch off into a cylindrical singularity or vanish in a spherical singularity in regions where the weak surgery flow is not present. From our canonical neighbourhood assumption, one expects the inner barrier to vacate the (ambient) interior of a $\delta$-neck in the weak surgery flow by translating like a bowl or passing through a singularity. Similarly, we expect the weak surgery flow would vacate the interior of a neck-like region in the outer barrier developing into a singularity by surgery \footnote{that is, of course, presuming that surgery is permitted according to the Definition \ref{brokenflow}.}. Indeed, this seems to indicate a correspondence of surgeries and singularities and thus one expects, along the sequence of weak surgery flows from $M$, for surgeries to accumulate in the singular set of $\mathcal{M}$.

Unfortunately, it is not clear that this picture is entirely correct. One possible issue is that there is no way to rule out a surgery neck developing in a weak surgery flow in such a way that is completely unrelated to the geometry of the barriers flows. This is possible as we have only shown the weak surgery flows (with large $\hthic$) remain Hausdorff-close to the original weak flow after the first surgery time. 
For the above heuristic to have rigorous meaning we need to be able to relate the geometry of the weak surgery flows back to that of the original flow. Indeed, this would rule out ‘gratuitous’ surgery necks forming in regions where we would expect low curvature.  One might hope to use pseudolocality to control the flow with surgery. Unfortunately, direct application of pseudolocality is obstructed by the surgeries, as the caps cannot be written as graphs over the necks they replace.  We will show in Proposition \ref{surgerypseudo} that the pseudolocality result as stated in \cite{pseudo} can be applied at a space-time point $X_0$ in a weak flow with surgery, with the caveat that surgeries must be performed at a scale much larger than the curvature at the point $X_0$. In order to repeatedly apply pseudolocality one must introduce further ingredients (see Remark \ref{pseudosnag}).

The purpose of the following lemma, Lemma \ref{lazybound}, is to define a scaling factor $\lambda := \frac{|A|(\mathbf{x}_0)}{\tilde{C}_2}$, such that when the flow is dilated by $\lambda$, the hypotheses of the pseudolocality, Theorem \ref{pseudo}, are satisfied, see Remark \ref{whyuseful}.

\begin{remark}
In Remark \ref{reasonforcontrol}, we discussed how the canonical neighbourhoods had to be chosen careful such that we always satisfy the hypotheses of the Haslhofer--Kleiner curvature estimates, Theorem \ref{hkcurv}, in the interior for a particular choice of $\Lambda$. We now pause to start fixing the value of our constants so we can use them in the following arguments. In particular, we fix a value for the required $\Lambda$.
\end{remark}

 We fix $\eta>0$ that satisfies the required gradient bound of the Ecker--Huisken graphical curvature estimate, Theorem \ref{ehest}. Taking this value of $\eta$ into Pseudolocality, Theorem \ref{pseudo}, fixes an initial Lipschitz bound $\varepsilon=\varepsilon(n,\eta)>0$ and radius $\vartheta=\vartheta(n,\eta)>0$. We hence take $\vartheta$ as the radius of the n-ball in the Ecker--Huisken estimate, Theorem \ref{ehest}, giving the constant $\tilde{C}_3=\tilde{C}_3(n,\theta,\vartheta)$. We will only ever apply this graphical curvature bound to a point over the origin of the ball, so the value of $\theta$ does not matter, so for the sake of simplicity take $\theta=1/2$. We can now fix $\Lambda = 10n\max\{\tilde{C}_3,1\}$ for application of the Haslhofer--Kleiner curvature estimate. As was discussed in Remark \ref{reasonforcontrol}, the value of $\Lambda$ needs to be fixed so it is certain we can apply the estimate at interior points of $\Omega_{(\alpha,\beta)}$. The reasoning for this choice of value for $\Lambda$ will become clear in the following theorems. Of course, fixing the value of $\Lambda$ fixes the value of $\tilde{C}_0=\tilde{C}_0(\alpha,\Lambda)<\infty$, the constant from the Haslhofer--Kleiner curvature estimate. Finally, taking  $\varepsilon$ given to us from pseudolocality and this value of $\tilde{C_0}$, we fix the value of $\tilde{C}_2=\varepsilon/\tilde{C}_0$, as per Lemma \ref{lazybound}.

In the following, constants will be denoted $\tilde{C}_k$ for some integer $k$ and cylinders\footnote{The cylinder has been defined in the appendix} will be denoted $\mathcal{C}_{r}$ for some radius $r>0$. Note also balls in the $(n+1)$-dimensional ambient space are denoted $B$, whilst balls of dimension $n$ in an affine subspace (i.e. a tangent space) will be denoted $B^n$.

\begin{lemma}\label{lazybound}
Let $\surg$ be a weak flow with surgery and suppose $X_0=(\mathbf{x}_0,t_0)\in\surg\cap\overline\Omega_{(\alpha,\beta)}$. Suppose further $t_0\leq t_F$, where $t_F$ is the last surgery time. 

For every $\varepsilon>0$, let $\tilde{C}_2(\alpha,\Lambda,\varepsilon)=\frac{\varepsilon}{\tilde{C}_0(\alpha,\Lambda)}$, where $\tilde{C}_0$ is the constant from the Haslhofer--Kleiner curvature estimate. Then the hypersurface $\lambda (M_{t_0}-\mathbf{x}_0)$, with $\lambda = \frac{|H|(\mathbf{x}_0)}{\tilde{C}_2}$, has 
\begin{align}
    \sup_{\lambda M_{t_0} \cap{B(\mathbf{0},1)}}|A|&\leq \varepsilon\\
    \sup_{\lambda M_{t_0} \cap{B(\mathbf{0},1)}}\sqrt{1+|Du|^2} &< 1+\varepsilon
\end{align}
Where $u(\mathbf{x})$ is a function on the tangent space at $0$ such that $\lambda (M_{t_0}-\mathbf{x}_0)\cap \mathcal{C}_1(\mathbf{0})=\mathrm{graph}(u)$ and $M_{t_0}$ is the $t=t_0$ time-slice of $\surg$.
In particular, we note that the above show that the Lipschitz constant of $u$ is bounded by $\varepsilon$.
\end{lemma}
\begin{proof}Since $t_0\leq t_F$, the surgery flow is certainly smooth, and thus we can apply the global curvature estimate, Theorem \ref{hkcurv}, with our choice of $\Lambda\geq 1$. The claim follows immediately.
\end{proof}

\begin{remark}\label{whyuseful} The existence of such a $\tilde{C}_2$ is noteworthy, as it is uniform across any $(\alpha,\delta)$-flow that satisfies the assumptions of Theorem \ref{hkcurv}. Indeed, this shows that the $\vartheta>0$ given to us in the following pseudolocality theorem (Theorem \ref{surgerypseudo}) is uniform, when working at the scale of mean curvature, across all weak surgery flows $\surg$ that satisfy the hypotheses of Theorem \ref{surgerypseudo}. This is required so limits may be taken.
\end{remark}

As mentioned previously, the surgeries obstruct the use of pseudolocality as stated in \cite{pseudo}. Following their argument, the result is only valid until the next surgery is performed. In addition to their proof, we need to show that if any surgeries are performed in the forward time interval, then they are not performed in or near a large neighbourhood of the cylinder where we wish to apply pseudolocality. Indeed, this is true provided surgeries are done at a sufficiently large scale compared to the mean curvature of the point we wish to apply pseudolocality. The central idea is a combination of the Ecker--Huisken graphical curvature estimates and the Haslhofer—-Kleiner curvature estimate to bound the mean curvature in the cylinder below the surgery scale.

\begin{prop}\label{surgerypseudo}
   Let $X_0\in \surg \cap \Omega_{(\alpha,\beta)}$, $|A|(X_0)<\infty$. Pseudolocality can be applied to the flow $\surg$ around $X_0$, provided the surgery is done with parameter $\hneck>\frac{\tilde{C}_0\tilde{C}_3}{\tilde{C}_2}n^2 |H|(X_0)$. That is,
   \begin{align}
        \mathcal{D}_\lambda(\surg-X_0)\cap \mathcal{C}_{\vartheta}(\mathbf{0}), t\in [0,\vartheta^2)\cap[0,t_F]
   \end{align}
   is a smooth mean curvature flow, and can be written as a graph over $B^n_{\vartheta}$ with Lipschitz constant less than $\eta$ and height bounded by $\eta\vartheta$. $\lambda=\lambda(\alpha,\Lambda,\varepsilon,X_0)$ is as in the above claim. $t_F$ denotes the final time of surgery in the dilated flow. Moreover, since $\surg$ is continued as a Brakke flow after the final time of surgery, we also deduce
   \begin{align}
        \mathcal{D}_\lambda(\surg-X_0)\cap \mathcal{C}_{\vartheta}(\mathbf{0}), t\in [0,\vartheta^2)\cap[0,t_\mathrm{Ext}]
   \end{align}
   is a unit-regular, cyclic (mod 2), integral Brakke flow, and can be written as a graph over $B^n_{\vartheta}$ with Lipschitz constant less than $\eta$ and height bounded by $\eta\vartheta$.
   
\end{prop}

\begin{remark}
$\tilde{C}_0, \tilde{C}_3$ are expected to be large, $\tilde{C}_2$ is expected to be small. Thus, $\frac{\tilde{C}_0\tilde{C}_3}{\tilde{C}_2}$ is very large. This may give the impression that the theorem is weak. Its strength will come once applied to points with bounded curvature in a sequence of flows with degenerating surgery parameters. 
\end{remark}

\begin{proof} Suppose $X_0\in \surg \cap \Omega_{(\alpha,\beta)}$, $|A|(X_0)<\infty$.
We fix $\eta>0$, and let $\vartheta(\eta),\varepsilon(\eta)$ be those given by the pseudolocality Theorem \ref{pseudo}. Let $\lambda$ be as in lemma \ref{lazybound} with  $\varepsilon= \varepsilon(\eta)$.

 If the surgery flow is a smooth mean curvature flow in the forward time interval given by Theorem \ref{pseudo}, then there is nothing to check. Thus, let $\hat{\surg}=\mathcal{D}_\lambda(\surg-X_0)$, and suppose there are surgeries occurring in the time interval $[0,\vartheta^2)$. Note, there are only finitely many times to check in this interval, so we may enumerate them chronologically.

Let $t_1$ be the time of the first surgery in $\hat{\surg}$ after time $t=0$. It is sufficient to show that all surgeries are performed far from the set $\mathcal{C}_{1}(0)$ at time $t_1$, as this demonstrates the flow is simply a smooth mean curvature flow in $\mathcal{C}_{1}(0)\times[0,t_2)$ and thus the flow remains a graph in the cylinder $\mathcal{C}_{\vartheta}(0)\times[0,t_2)$, where $t_2$ is the next surgery time. 
\begin{remark}
These times correspond to surgeries in the dilated flow, not the original time scale.
\end{remark}

Since the flow is a mean curvature flow on $[0,t_1]$, we know from the classical pseudolocality result that $
    \hat\surg\cap \mathcal{C}_{\vartheta}(\mathbf{0})$ can be written as the graph of $u_t: B^n_{\vartheta}(0)\to \mathbb{R} $, for $ t\in [0,\vartheta^2)\cap[0,t_1]$.

Applying the Ecker--Huisken interior estimate for graphs, Theorem \ref{ehest}, to the function $u_t$ we establish the following bounds on curvature \begin{align}\label{goodbound}
    \sup_{B^n_{\theta\vartheta} (0)\times [0,t_1]}|A| \leq\tilde{C}_3(n,\theta,\vartheta)\sup_{B^n_{ \vartheta} (0)\times \{0\}} |A|=\tilde{C}_3 \varepsilon
\end{align}
for some constant $\tilde{C}_3$ depending only on $n,\theta,\vartheta$.

Let $X=(\mathbf{0},u_{t_1}(0),t_1)=(\mathbf{x},t_1)$, i.e.~ the point in the flow above the origin at time $t_1$. Equation \ref{goodbound} shows $|A|(X)\leq \tilde{C}_3 \varepsilon$. Applying the Haslhofer--Kleiner curvature estimate, Theorem \ref{hkcurv}, at the point $X$, we deduce that in the backward parabolic cylinder $P(X, \Lambda r)$ the curvature is bounded by $\tilde{C}_0 r^{-1}$, where $r^{-1}=
H(X)\leq \tilde{C}_3 \varepsilon n$ (and thus, $r\geq(\tilde{C}_3\varepsilon n)^{-1}$). Note we have used the standard inequality $|H|\leq n  |A|$.

As a simple consequence of the estimate in $P(X,\Lambda r)$, we have 
\begin{align*}
    \sup_{B_{\Lambda r}(\mathbf{x})\cap \hat{M}_{t_1}}|A|\leq\tilde{C}_0\tilde{C}_3\varepsilon n,
\end{align*} where $\hat{M}_{t_1}$ denotes the $t=t_1$ time slice of $\hat{\surg}$. Moreover, using $|H|\leq n |A|$ once again, we see 
\begin{align*}
    \sup_{B_{\Lambda r}(\mathbf{x})\cap \hat{M}_{t_1}}|H|\leq\tilde{C}_0\tilde{C}_3\varepsilon n^2.
\end{align*} We highlight that, since $\Lambda\geq10n\tilde{C}_3$, the curvature bound holds in $B(\mathbf{x},10\varepsilon^{-1})$, moreover $B(\mathbf{x},10\varepsilon^{-1})\supset\mathcal{C}_1(0)$. That is to say, the curvature bound holds for the weak flow with surgery contained in the cylinder $\mathcal{C}_1(0)$ at time $t_1$. 

By definition, surgery in $\surg$ was done at scale $\hneck$. Scaling our parameters accordingly, we deduce surgery in $\hat\surg$ is done at scale $\hat{H}_\mathrm{neck}=\lambda^{-1} \hneck=(\tilde{C}_2/|H|(x_0))\hneck>\tilde{C}_0\tilde{C}_3n^2>\tilde{C}_0\tilde{C}_3\varepsilon n^2$. Here, we have used our assumption that $\hneck>\frac{\tilde{C}_0\tilde{C}_3}{\tilde{C}_2}n^2 |H|(X_0)$ and that $\varepsilon<1$. Observe, from the bound on mean curvature in $B_{\Lambda r}(\mathbf{x})$, the mean curvature at every point $Y\in\hat\surg\cap( B_{10\varepsilon^{-1}}(x)\times\{t_1\})$ is below the threshold for surgery to be performed. In particular, any changes made at time $t_1$ do not affect the portion of the hypersurface $\hat{M}_{t_1}$ contained in $\mathcal{C}_1(0)$.  Hence, the flow $\hat\surg\cap(\overline{\mathcal{C}}_1(0)\times[0,t_2])$ is a smooth mean curvature flow, and the flow is graphical over $B^n_\vartheta(0)$ in $\mathcal{C}_\vartheta(0)\times[0,t_2]$.

This argument is then repeated at all future surgery times in $[0,\vartheta^2)\cap[0,t_F]$. The second claim follows immediately from the Brakke form of Theorem \ref{pseudo}, as $\surg$ is continued as a unit-regular integral Brakke flow after the final surgery time $t_F$.
\end{proof}

We now have the tools necessary to show surgeries accumulate in the singular set.

\begin{prop}\label{surgcontrol}
Let $M^n\subset\mathbb{R}^{n+1}$ and $\mathcal{M}$ be as above. Then, for every open neighbourhood $N$ of the singular set, there is a $H_{min}(N)<\infty$ such that if $\mathbb{H}$ has $\hthic>H_\mathrm{min}$, then all surgeries in $\surg$ occur inside this neighbourhood. 
\end{prop}
\begin{proof}
 The above statement is equivalent to the statement that, across a sequence of surgery flows with $H^i_\mathrm{th}\to \infty$, any sequence of centres of surgery necks, $X_i\in\surgi$, accumulates in the singular set $\mathfrak{S}$ of $\mathcal{M}$. 
 
 Suppose for contradiction that this is not the case. Let  $\mathcal{M}_{\mathbb{H}_i}$ be a sequence of $(\Balpha,\delta,\mathbb{H}_i)$-flows evolving from $M$ with $H^i_{th}\to \infty$. By the assumption we wish to contradict, we can find a sequence of points $X_i=(\mathbf{p}_i,t_i)\in\mathcal{M}_{\mathbb{H}_i}$ in $\delta$-necks where surgery is performed, with $H(X_i)=H^i_\mathrm{neck}$, that accumulate to some point $X_\infty=(\mathbf{x}_\infty,t_\infty)\in \mathfrak{S}^c$. It is clear that the sequence must accumulate to some point in $\mathrm{supp}(\mathcal{M})$ from Hausdorff convergence. Note that $t_\infty \neq t_\mathrm{Ext}$, as the regular set is empty at time of extinction. 
 
 \begin{claim} \label{graph}
 $X_\infty\notin \partial\Omega_{(\alpha,\beta)}$
 \end{claim}
 
 \begin{proof}
 
 Suppose $X_\infty$ were in the boundary of the chosen $(\alpha,\beta)$-neighbourhood. Item (iii) of Definition \ref{brokenflow} required a backward parabolic cylinder centred at each point in the boundary in which the surgery flow is a graph over the original flow. 
 This immediately rules out surgeries being performed in this neighbourhood, and thus preventing accumulation forward in time (i.e. $t_i<t_\infty$, for infinitely many $i$) or `spatially' ($t_i=t_\infty$, for infinitely many $i$) within a given time-slice to a point the boundary. Thus, it remains to check that surgeries cannot accumulate backward in time ($t_i>t_\infty$, for infinitely many $i$) to a point in the boundary.
 
 We first prove a smooth convergence result. Again we recall Item (iii) of Definition \ref{brokenflow}. There is a backwards parabolic cylinder $P=P(X_\infty, 2\xi H_\mathrm{bdd})$ centred at $X_\infty$ in which we can write $\surgi$ as a graph over $\mathcal{M}$. This is true for all $i$. As mentioned above, being a small graph over the original flow rules out surgeries occurring in this parabolic cylinder. Clearly $\surgi\cap P$ is a sequence of smooth unit-regular Brakke flows, and thus converge to some limiting Brakke flow $\mathcal{N}$ in $P$. Hausdorff convergence shows that the support of $\mathcal{N}$ is  $\textrm{supp}(\mathcal{M}\cap P)$. Finally, we note that being a small graph controls the multiplicity of the flows with surgery and thus the sequence converges locally smoothly in $P$ to $\mathcal{M}\cap P$ by White regularity.
 
 The smooth convergence is now used to show pseudolocality can be applied in such a way that is comparable across all the flows with surgery for sufficiently large $i$.  
 Dilating by $\lambda=|H|(X_\infty)/\tilde{C}_2$ around the point $X_\infty$, and following the proof of Lemma \ref{lazybound}, we deduce $\tilde{M}_0$, the $t=0$ time slice of the dilated flow $\tilde{\mathcal{M}}=\mathcal{D}_\lambda(\mathcal{M}-X_\infty)$, can be written as the graph of some smooth function $u$ over $B=B^n_1(0)$, the ball of radius 1 in the tangent space at $0$, with $|A|<\varepsilon$. Similarly, we set $\tilde{\mathcal{M}}_i=\mathcal{D}_{\lambda_i}(\mathcal{M}_{\mathbb{H}_i}-X_\infty)$, $\lambda_i=|H_i|(X_i)/\tilde{C}_2$. Since the (un-dilated) flows converged smoothly around $X_\infty$, we deduce $\lambda_i\to\lambda$. Moreover, the dilated flows $\tilde{\mathcal{M}}_i$ converge smoothly to $\tilde{\mathcal{M}}$ in $P$, thus there is an $I<\infty$ such that for $i\geq I,$ the time $t=0$ time-slice, $\tilde{M}_{i,0}$, can be written as a graph of the function $u_i:B\to\mathbb{R}$, where $B$ is the same ball in the tangent space to $\tilde{M}_0$ at $0$, and $u_i\to u$ smoothly in $B$. Thus, by the Brakke form of the pseudolocality result for flows with surgery, Proposition \ref{surgerypseudo}, no surgeries of the flow $\surgi$ occur in $\tilde{\surgi}\cap{\mathcal{C}_{\vartheta_i}(0)}\times([0,\vartheta_i^2)\cap[0,t_\mathrm{Ext}) )$. Recall, $\vartheta_i$ essentially depended on the curvature at $u_i(0)$ and the dimension. Since the hypersurfaces at time $t=0$ converge smoothly in some neighbourhood of the origin, there is a uniform $\vartheta>0$ such that for every flow, $\tilde{\mathcal{M}}_i\cap{\mathcal{C}_\vartheta(0)}\times([0,\vartheta^2)\cap[0,t_\mathrm{Ext}))$ is a unit-regular, cyclic (mod 2) Brakke flow. In particular, no surgeries occur. This contradicts our assumption that surgeries were accumulating from future times. 
 \end{proof}

It remains to check regular points in the interior  of $\Omega_{(\alpha,\beta)}$. In order to employ the above argument, we require knowledge that the weak surgery flows are graphical over $\mathcal{M}$ in some backwards parabolic cylinder. A priori, we have no control of the flow at points in the interior, other than information given by the maximum principle and Hausdorff convergence. To find such a neighbourhood, we will start at the boundary of $\Omega_{(\alpha,\beta)}$ and then repeatedly apply the pseudolocality theorem followed by the Haslhofer--Kleiner curvature estimate to work our way into the interior.

\begin{claim}
There is an open space-time neighbourhood of $X_\infty$ such that the flows $\surgi$ converge smoothly to $\mathcal{M}$.
\end{claim}

\begin{remark}\label{pseudosnag}
  If one were to just iterate pseudolocality, the forward time interval could shrink in a geometric progression.  The essence of the argument presented below is, given a point of low curvature, we find our forward neighbourhood from pseudolocality. We deduce convergence of the sequence of surgery flows to $\mathcal{M}$ in this forward neighbourhood. Applying the Haslhofer--Kleiner curvature estimate we show, for large $i$, no surgeries will be performed in a larger backward neighbourhood (centred at some future time, compared to the point we applied pseudolocality), and we can deduce convergence on this larger set. One is then in a position to apply pseudolocality at the same scale. 
\end{remark}
\begin{proof}
Consider a path $\gamma$ in $\mathrm{reg}(\mathcal{M})\cap \Omega_{(\alpha,\beta)}$ connecting $X_\infty$ to a point $X_0\in\partial \Omega_{(\alpha,\beta)}$. Say $\gamma:[0,T]\to \mathrm{reg}(\mathcal{M})$, $\gamma(0)=X_0, \gamma(T)=X_\infty$.
Since the flow is locally 2-convex, we can pick the point $X_0$ and translate in time such that $X_0=(\mathbf{x}_0,0)$, $\gamma(\tau)\in M_\tau$. We will write $\gamma(\tau)=(\mathbf{x}_\tau,\tau)$. The argument proceeds as follows:
\begin{itemize}
     \item Since the path $\gamma$ is compact, there exists some $\mathcal{A}<\infty$ such that \begin{align*}\max_{\tau\in[0,T]}|H_{\mathcal{M}}|(\gamma(\tau))\leq\mathcal{A}.\end{align*}
    
    \item Fix a small constant $\zeta>0$. Lemma \ref{lazybound} 
    implies $\tilde{M}_\tau=\mathcal{D}_\lambda(M_\tau-\gamma(\tau))$ can be written in $\mathcal{C}_1(0)$ as a graph over the ball $B^n_1(0)$ in the tangent space to $\tilde{M}_\tau$ at $0$, where $\lambda = \frac{\mathcal{A}+\zeta}{\tilde{C}_2}$. In particular, the hypotheses of Theorem \ref{pseudo} are satisfied and hence we can apply the Brakke formulation of pseudolocality to $\tilde{M}_\tau$ at $0$.
    \item We remark that the small constant $\zeta>0$ is present so we can rescale each $\surgi$ by the same factor. The plan is to use the same argument as in Claim \ref{graph}, with the only complication coming from wanting to have the forward neighbourhood be comparable at every point along $\gamma$. Consider a sequence of points $Y_i\in\surgi$ accumulating to $Y_\infty\in\gamma$, such that $|H_{\surgi}(Y_i)|\to |H_{\mathcal{M}}(Y_\infty)|$. Then, there exists an $I=I(\zeta)$, such that $i\geq I$ implies $|H_{\surgi}(Y_i)|<\mathcal{A}+\zeta$. The significance being one can choose a cylinder centred at $Y_\infty$ in which the conclusion of pseudolocality (Theorem \ref{pseudo} and Proposition \ref{surgerypseudo}) is valid for $\mathcal{M}$ and all $\surgi$ with $i\geq I(\zeta)$ after dilating by the common constant $\lambda=\frac{\mathcal{A}+\zeta}{\tilde{C}_2}$. 
    \item Returning to our main argument, we transform back to the un-dilated flow and deduce there is a uniform $\vartheta$ such that at each point $\gamma(\tau)\in\mathrm{reg}\mathcal(M)$, the flow $\mathcal{M}\cap\mathcal{C}(\tau)$ is graphical over the ball $B^n_{\lambda^{-1}\vartheta}(x_\tau)$ in the tangent space to $M_\tau$ at $\gamma(\tau)$. Where $\mathcal{C}(\tau)=\mathcal{C}_{\lambda^{-1} \vartheta}(x_\tau)\times([\tau,\tau+(\lambda^{-1}\vartheta)^2]\cap[0,T_\mathrm{Ext}))$.
    \item The path $\gamma$ is continuous and compact. Hence, we can find finitely many times $0=\tau_0<\tau_1< \cdots< \tau_N<T$ such that $\gamma([0,T])\subset\cup^N_{j=0}\mathcal{C}(\tau_j)$. Note that $\tau_N<T$. This will be important for applying the curvature estimates to the flows with surgery $\surgi$. Note further there must be `overlap' of the cylinders, in the sense $\gamma(\tau_i)\in \mathcal{C}(\tau_{i-1}), i\geq 1$.
    
    \item By our assumption, $\gamma(0)\in\partial\Omega_{(\alpha,\beta)}$. Examining the proof of Claim \ref{graph}, we can immediately deduce Brakke convergence of $\surgi \to \mathcal{M}$ in $\mathcal{C}(0)$. Indeed, for sufficiently large $i$, $\surgi\cap\mathcal{C}(0)$ is a Brakke flow (no surgeries occur in $\mathcal{C}(0)$). Multiplicity is controlled by our assumption the flows with surgery are graphical over the boundary.
    
    \item We can improve the regularity of the convergence. Recall, $\gamma(\tau_1)\in\mathcal{C}(0)$, thus $\surgi\to\mathcal{M}$ in a Brakke sense in some small backwards parabolic cylinder $P$ centred at $\gamma(\tau_1)$. We may suitably shrink $P$ such that $P\cap\mathcal{M}\subset\mathrm{reg}(\mathcal{M})$. Since $\mathcal{M}$ is smooth in $P$, we deduce smooth convergence of $\surgi\to\mathcal{M}$ in $P$ by White regularity. 
    
    \item We now prove an inductive step, allowing us to `move along' the path $\gamma$. Smooth convergence in $P$ centred at $\gamma(\tau_1)$ implies there is a sequence of points $Y_i=(\mathbf{y}_i,\tau_1)\in\surgi, Y_i\to \gamma(\tau_1), H_{\surgi}(Y_i)\to H_\mathcal{M}(\gamma(\tau_1))$. We can hence apply the Haslhofer--Kleiner curvature estimate to $\surgi$ at $Y_i$ as in Proposition \ref{surgerypseudo} to deduce no surgeries occur in the backwards parabolic cylinder $P(Y_i, \Lambda (H_{\surgi}(Y_i))^{-1})$. Applying the curvature estimate is permissible when $i$ is taken sufficiently large: the surgery necks accumulate at some time $T$ with $\tau_1<T$, thus for large $i$ we must have $\tau_1<t_{F_i}$, where $t_{F_i}$ is the final time of surgery in $\surgi$.
    \item In particular, we deduce smooth convergence in $P(\gamma(\tau_1, (H_{\mathcal{M}}(\gamma(\tau_1)))^{-1})$ since $\Lambda>1$. One is now in the position to apply the argument from Claim \ref{graph}.
    
    \item This argument can be repeated at each $\tau_j$, since $\tau_j\in\mathcal{C}(\tau_{j-1})$. In particular, we note that $\gamma(T)\in\mathcal{C}(\tau_N)$. Thus, again, taking $i$ sufficiently large, we deduce no surgeries of the flow $\surgi$ are performed near $\gamma(T)$, contradicting the claim that surgeries accumulated at $\gamma(T)=X_\infty$.
    
   \end{itemize}
   \end{proof}
    
   This concludes the proof, as we have shown surgeries cannot accumulate to regular points.

 \end{proof}

 We now state and prove our crucial convergence result. Note, in the above proof we have already established convergence inside $\Omega_{(\alpha,\beta)}$.
 
 \begin{prop}[Convergence away from singular set]\label{convergence2} Let $\mathcal{M}_{\mathbb{H}_i}$ be a sequence of $(\Balpha,\delta,\mathbb{H}_i)$ surgical flows derived from $M$, and suppose $\mathbb{H}_i$ is a sequence of surgery parameters with $H^i_{th}\to \infty$. Then, $\mathcal{M}_{\mathbb{H}_i}$ converges to $\mathcal{M}$ as Brakke flows on the complement of the singular set of $\mathcal{M}$.
 \end{prop}

 \begin{proof} Recall that the singular set $\mathfrak{S}$ is closed in space-time, thus its complement, $\mathfrak{S}^c$, is open. Recall further, the definition of convergence of Brakke flows \cite{Brakke}, \cite{ilmanen}, is with respect to compactly supported functions. If $f\in C^1_c(\mathfrak{S}^c)$, then by definition we have $\textrm{supp}(f)\Subset \mathfrak{S}^c$. In particular, it is sufficient to verify the proposition on any connected open set $\Omega\Subset\mathfrak{S}^c$ that has non-trivial intersection with the initial timeslice. These properties are required to control the multiplicity of the Brakke flow as in Lemma \ref{convergence}.


 \begin{claim}
  For any open set $\Omega \Subset \mathfrak{S}^c$, there is an $I<\infty$ such that for $i>I$, no surgeries of the flow $\mathcal{M}_{\mathbb{H}_i}$ occur in $\Omega$.
 \end{claim}
 \begin{proof}
This follows from Proposition \ref{surgcontrol}.
 
 If $\Omega\cap\Omega_{(\alpha,\beta)}=\emptyset$, we immediately know surgeries are not present in a neighbourhood for all $i>0$. It remains to check the case when $\Omega \cap\Omega_{(\alpha,\beta)}\neq \emptyset$. Without loss of generality, we consider $\Omega \subset \Omega_{(\alpha,\beta)}$. Since $\Omega \Subset \mathfrak{S}^c$, there is an open neighbourhood $N$ of $\mathfrak{S}$, with $\Omega\cap N= \emptyset$.
 
 Thus, by Proposition \ref{surgcontrol} we deduce all surgeries occur in $N$ for sufficiently large $i$, and hence none occur in $\Omega$.
 \end{proof}

   Applying Ilmanen's compactness result for Brakke flows, \cite{ilmanen}, there is a limiting unit-regular Brakke flow $\mathcal{N}$ such that, 
  \begin{align*}
  \lim_{i\to\infty}\mathcal{M}_{\mathbb{H}_i}\lfloor \Omega = \mathcal{N}.
  \end{align*}
 
 \begin{claim}\label{supportequal}
  $\textrm{supp}(\mathcal{N})=\textrm{reg}(\mathcal{M})\cap\Omega$
 \end{claim}
 \begin{proof}
 The claim follows immediately from Corollary \ref{Hausdorff}. In particular, $\textrm{supp}(\mathcal{N})$ is connected by the result of \cite{ccms20}.
  \end{proof}
  \begin{claim}
   $\mathcal{N}=\mathcal{M}\lfloor{\Omega}$ as unit-regular Brakke flows.
  \end{claim}
  \begin{proof}
    All that remains is to check $\mathcal{N}$ does not develop higher multiplicity. By the above, $\mathrm{supp}(\mathcal{N})$ is connected and has non-trivial intersection with the initial time-slice, thus $\mathcal{N}$ has unit density everywhere.
 \end{proof}
Thus, $\lim_{i\to\infty}\mathcal{M}_{\mathbb{H}_i}\lfloor \mathfrak{S}^c = \mathcal{M}$ as Brakke flows.

\end{proof}

As a corollary, one deduces the following results that control the behaviour of any potential singular points that form in weak surgery flows.
 
 \begin{corollary}
    
\label{smoothpoints}
 Let $\mathcal{M}_{\mathbb{H}_i}$ be a sequence of $(\Balpha,\delta,\mathbb{H}_i)$ surgical flows derived from the flow $\mathcal{M}$, and suppose $\mathbb{H}_i$ is a sequence of surgery parameters with $H^i_\mathrm{th}\to \infty$. If $X_i\in\mathcal{M}_{\mathbb{H}_i}$ is a sequence of singular points (i.e.~ points with Gaussian density $\Theta_{\mathcal{M}_{\mathbb{H}_i}}(X_i) \geq 1+\varepsilon_{\text{White}}$). Then $X_i$ accumulate in $\mathfrak{S}$, the singular set of $\mathcal{M}$.
 \end{corollary}
 \begin{remark}
  Here $\varepsilon_\text{White}$ is the (dimension dependent) quantity of White regularity \cite{bw2005}. 
 \end{remark}
 \begin{proof}

 Suppose for contradiction a sequence of points $\{X_i\}_i^\infty$, satisfying the above hypothesis, accumulates at $X_\infty\in \textrm{reg}(\mathcal{M})$. 
 Then, by Proposition \ref{convergence2}, the weak surgery flows converge to $\mathcal{M}$ in a neighbourhood of $X_\infty$. In particular, $\Theta_{\mathcal{M}}(X_\infty)=1$.
 This is in contradiction to the upper semi-continuity of the density; taking the limit of densities we should have $\Theta_{\mathcal{M}}(X_\infty)\geq 1+\varepsilon_\text{White}$.

 \end{proof}
 
 \begin{corollary}\label{bigA}
 The above corollary holds also for regular points $X_i\in \mathcal{M}_{\mathbb{H}_i}$ where $$\lim_{i \to \infty}|A(X_i)|= \infty$$
 \end{corollary}
 \begin{proof}
 Following the above proof, we note that smooth convergence implies convergence of the second fundamental form. $X_\infty$ is a smooth point, thus $|A|<\infty$, contradicting $\lim_{i \to \infty}|A(X_i)|\to \infty$.
 \end{proof}

 \section{Existence and Convergence for Smooth Mean Curvature flow with surgery}
Let $M^n\subset\mathbb{R}^{n+1}$ be a closed, smoothly embedded submanifold. Since $M$ is compact and smooth, we can find a $\gamma>0$ such that $|A|<\gamma$. We suppose there is a unique unit-regular Brakke flow $\mathcal{M}$ emerging from $M$ that encounters only spherical and neck-pinch singularities.  We fix
\begin{itemize}
     \item $0<\alpha<\textrm{min}\{\alpha_{cyl},\alpha_{sphere},\alpha_{oval},\alpha_{bowl}\}$.
     \item $0<\beta<\text{min}\{\beta_\text{sphere},\beta_\text{cylinder},\beta_\text{bowl},\beta_\text{oval}\}$.
 \end{itemize} Let $\Balpha=(\alpha,\beta,\gamma)$.
 Additionally, we take $\delta>0$ small enough that all the arguments of Haslhofer--Kleiner \cite{hksurg} hold and to satisfy item (iii) of Definition \ref{brokenflow} and Remark \ref{reasonforcontrol}. For the sake of completeness, we also fix a suitable standard surgical cap, suitable cap separation parameter and the value of $\Lambda$ as in Section 4.

 \begin{theorem}[Surgery at the first singular time] \label{firstsurg}
 Let $\mathcal{M}$ be as above. Let $\Omega_1$ be the union of the connected components of $\Omega_{(\alpha,\beta)}$ containing the first singular time. Let $T_1>0$ be the first singular time of the flow outside $\Omega_1$. Then for every $\varepsilon>0$, the parameters $H_\mathrm{min}(M)<\infty$ and $\Theta(M)<\infty$ can be chosen (depending only on the initial hypersurface) such that the $(\Balpha,\delta,\mathbb{H})$ weak surgery flow $\mathcal{M}_\mathbb{H}$ is a smooth mean curvature flow with surgery on $[0,T_1-\varepsilon)$.

 \end{theorem}
 
 Compare the result of Mramor, \cite{mramor19}, where similar ideas are discussed for surgery in mean convex `patches' of non-compact flows.

 \begin{proof} Fix an $\varepsilon>0$ and stipulate that surgeries may only be performed in $\Omega_1$.
By Corollary \ref{bigA}, we know the singularities of surgery flows converge to the singular set of $\mathcal{M}$ as $\hthic\to\infty$. Thus, we can choose $H_\mathrm{min}<\infty$ sufficiently large that all singularities of a weak surgery flow with $H_\mathrm{th}>H_\mathrm{min}$ occur within $\varepsilon$ in time of the singularities of $\mathcal{M}$. Moreover, such singularities are contained in $\Omega_{(\alpha,\beta)}$ and are spherical or neck-pinch singularities. 

We initially fix the surgery ratio $\Theta<\infty$, this will be changed in due course.
 \begin{claim}
 For sufficiently large $H_\mathrm{min}$, any $(\Balpha,\delta,\mathbb{H})$-flow with $H_\mathrm{th}>H_\mathrm{min}$ is a $\delta$-graph over $\mathcal{M}$ in $N_1$ along the boundary of $\Omega_1$.
 \end{claim}\begin{proof}
 This is a consequence of Proposition \ref{convergence2} and its corollaries. Recall, $N_1$ is the open neighbourhood of the boundary of $\Omega_1$ in which the flow $\mathcal{M}$ is smooth, locally $\alpha$-noncollapsed and $\beta$-uniformly 2-convex, as defined in Definition \ref{bddnbhd}. Since the boundary of $N_1$ is bounded away from the singular set, it is immediate from Proposition \ref{convergence2} and White regularity that, for sufficiently large $H_\mathrm{th}$, the claim holds.
 \end{proof}
 
 \begin{remark}
  It is important to compare this claim with the definition of surgery. We only permit the surgery procedure to be applied when the flow is graphically over $\mathcal{M}$ along the boundary. Thus, we see the obstruction to the flow continuing as a smooth surgery flow is not from our definitions, but from a point with $H(X)=\hneck$ that does not separate regions of curvature $\hthic$ and $\htrig$ or is not a $\delta$-neck. This is same obstruction as is dealt with in the case for 2-convex flows in \cite{hksurg}.
 \end{remark}
 \begin{claim}\label{surgtime}
  Fix $H_\mathrm{min}<\infty$ to satisfy claim 5.1. Then if $\hthic>H_\mathrm{min}$, we can directly apply the arguments of Haslhofer--Kleiner \cite{hksurg} to establish a $\Theta<\infty$ such that $\mathbb{H}>\Theta$ implies the weak $(\Balpha,\delta,\mathbb{H})$ surgery flow is a smooth mean curvature flow up to time $t=T_1-\varepsilon$. 
 \end{claim}
 \begin{proof}
Recall, the definition of an $(\alpha,\delta)$-Brakke flow only allowed surgery as long as the flow was smooth. Thus, since the singularities of the surgical flows can occur within $\varepsilon$ of any singular time, $T_1-\varepsilon$ is the best one can do without more information on the singular set.

By the first claim, $\mathcal{M}_{\mathbb{H}}\cap\partial \Omega_1$ is 2-convex and $\alpha$-noncollapsed for all $\mathbb{H}$ with $H_\mathrm{th}>H_\mathrm{min}$. After doing one surgical neck replacement, the maximum principle gives that the flow remains $2$-convex and $\alpha$-noncollapsed inside $\Omega_1$. The same argument holds across any number of neck replacements, so every surgical flow with $H_\mathrm{th}>H_\mathrm{min}$ is 2-convex and $\alpha$-noncollapsed inside $\Omega_1$.
 
  We now stipulate that the flow is stopped once $|H|=H_\mathrm{trig}$ is achieved inside $\Omega_1$. \cite[Theorem 1.21]{hksurg} and \cite[Theorem 1.22]{hksurg} can now be applied directly find the desired $\Theta<\infty$ which establishes the existence of a weak flow with surgery that is smooth inside $\Omega_1$ up to time $T_1-\varepsilon$. We note that Corollary \ref{bigA} prevents points of high curvature accumulating on the boundary of $\Omega_1$ along sequences of surgical flows. This is important for the proof of \cite[Theorem 1.22]{hksurg}.
 \end{proof}
 
 This completes the proof of the theorem.
\end{proof}
 
 \begin{remark} We stop only if $H_\mathrm{trig}$ is achieved in $\Omega_1$.
\end{remark}

\begin{remark}\label{andrewsdiscussion}
 One should note that Andrews' maximum principle proof of $\alpha$-noncollapsing for mean convex mean curvature flow, \cite{andrews}, makes use of a 2-point maximum principle for a function $Z(x,y,t)$. The positivity of $Z(x,y,t)$ is equivalent to being $\alpha$-noncollapsed. This argument can be suitably localised to the above situation by observing that along the boundary of $\Omega_1$, the flows will be close to one of the canonical flows (sphere, cylinder, bowl, and oval). Indeed, we know for points in the boundary the `touching points' of tangential spheres will be in our neighbourhood of the boundary, $N_1$. Since the interior mean curvature is larger than the boundary mean curvature, and surgery flows are Hausdorff close to the original flow, we see touching points of tangential spheres to interior points will be in $\Omega_1\cup N_1$. That is, one only needs to consider the function $Z(x,y,t)$ for points  $((\mathbf{x},t),(\mathbf{y},t)) \in\Omega_1 \times \{ \Omega_1\cup N_1\}$. This is similar to the argument presented in Theorem \ref{bddctrl}.
\end{remark}

 \begin{theorem}[Existence of a smooth flow with surgery]\label{existence} Let $\mathcal{M}$ be as above. Then, the parameters $H_\mathrm{min}(M)<\infty$ and $\Theta(M)<\infty$ can be chosen (depending only on the initial hypersurface) such that every weak $(\Balpha,\delta,\mathbb{H})$-flow, $\mathcal{M}_\mathbb{H}$, with $H_\mathrm{th}>H_\mathrm{min}$, $\mathbb{H}>\Theta$ satisfies:
 \begin{itemize}
     \item $|H|\leq H_\mathrm{trig}<\infty$ everywhere,  
     \item $\mathcal{M}_\mathbb{H}$ vanishes in finite time.
 \end{itemize}
 i.e.~ $\mathcal{M}_\mathbb{H}$ is a smooth mean curvature flow with surgery.
 \end{theorem}
 
 \begin{remark}
The weak surgery flows were unit-regular away from surgery, so sudden vanishing is not permitted. The second item is thus non-trivial.
 \end{remark}
 \begin{proof} $\Omega_{(\alpha,\beta)}$ has finitely many components, thus it is sufficient to argue inductively.

 We show that given Theorem \ref{firstsurg}, we have the respective statement for $\Omega_2$, the union of connected components of $\Omega_{(\alpha,\beta)}$ containing time $T_1$. Recall time $T_1$ was the first singular time that occurs outside $\Omega_1$. We will establish that for every $\varepsilon>0$ the parameters can be chosen such that there is a smooth flow with surgery up to time $T_2-\varepsilon$. Here, $T_2$ the first singular time outside of $\Omega_1\cup \Omega_2$ 
 
\begin{remark} The time interval over which $\Omega_2$ exists may overlap with that of $\Omega_1$. Surgeries in $\Omega_2$ can affect the surgeries that occur in $\Omega_1$, since mean curvature flow is parabolic. This is not an issue as the convergence results still hold. We may require a larger $H_\mathrm{min}$ and/or $\Theta$ for the same conclusion to hold.
\end{remark}
 
 Pick $H_\mathrm{min}$, $\Theta<\infty$ such that the conclusion of Theorem \ref{firstsurg} holds, and consider the boundary of $\Omega_2$. Once again, the logic of Proposition \ref{convergence2} controls the behaviour in a neighbourhood of the parabolic boundary, $N_2$. We may take $H_\mathrm{min}$ large enough that the flow is $\beta$-uniformly $2$-convex and $\alpha$-noncollapsed in $N_2$. Proceeding exactly as in claim \ref{surgtime}, we conclude the same result for $\Omega_1\cup\Omega_2$.

 This argument can be repeated for each connected component of $\Omega_{(\alpha,\beta)}$. Since there are only finitely many components, $H_\mathrm{min}$ and $\Theta$ stay bounded as they can only be changed a finite number of times. 
 
 Observe, the flow $\mathcal{M}$ will be entirely contained within the final connected component of $\Omega_{(\alpha,\beta)}$. Thus, there will be no singular times outside the final connected component, as there is no flow. The flow inside this final component will be a 2-convex surgery of \cite{hksurg}.
\end{proof}

We restate the canonical neighbourhood theorem of Haslhofer--Kleiner.
 
\begin{theorem}[Canonical Neighbourhood Theorem, Theorem 1.22 \cite{hksurg}]
For all $\varepsilon>0$, there exist $\overline{\delta}=\overline{\delta}(\Balpha)>0$, $H_{\mathrm{can}}(\varepsilon)=H_{\mathrm{can}}(\Balpha,\varepsilon)<\infty$ and $\Theta_\varepsilon(\delta)=\Theta_\varepsilon(\Balpha,\delta)<\infty$ ($\delta\leq\bar{\delta}$) with the following significance.
If $\delta\leq\overline{\delta}$ and $\mathcal{M}_\mathbb{H}$ is an $(\Balpha,\delta,\mathbb{H})$-flow with $\mathbb{H}\geq \Theta_\varepsilon(\delta)$,
then any $(\mathbf{p},t)\in \mathcal{M}_\mathbb{H}$ with $|H(\mathbf{p},t)|\geq H_{\mathrm{can}}(\varepsilon)$ is $\varepsilon$-close to either
(a) a $\beta$-uniformly $2$-convex ancient $\alpha$-noncollapsed flow,
or (b) the evolution of a standard cap preceded by the evolution of a round cylinder.
\end{theorem}
\begin{proof}
The proof is identical to that of Haslhofer--Kleiner \cite{hksurg}, for we only do surgery in 2-convex connected components.
\end{proof}
 The canonical neighbourhood theorem gives the following topological result concerning the dropped components. 
\begin{theorem}[Discarded components, \text{\cite[Corollary 1.25]{hksurg}}]\label{dropped}
For all $\varepsilon>0$ small enough, there are parameters $\Theta_\varepsilon(\delta)<\infty$, $H_\mathrm{can}(\varepsilon)$ such that any weak $(\Balpha,\delta,\mathbb{H})$ surgical flow with $\mathbb{H}>\Theta_\varepsilon(\delta)$, and $H_\mathrm{th}>H_\mathrm{can}(\varepsilon)$, has all discarded components are diffeomorphic to $\mathbb{S}^{n}$ or $\mathbb{S}^{n-1} \times \mathbb{S}^{1}$.
\end{theorem}
\begin{remark}
 The parameters are derived from the canonical neighbourhood theorem.
\end{remark}
 \begin{proof}
 This follows from the canonical neighbourhood theorem \cite[Theorem 1.22]{hksurg}. The argument is identical to that in \cite{hksurg}, for components are only dropped if they are contained in $\Omega_{(\alpha,\beta)}$.
 \end{proof}
 We conclude with a result similar to that of Lauer and Head, \cite{lauer, head}. Note we also establish the stronger result that the convergence away from the singular set is smooth.
  \begin{theorem}
 Taking the limit as $H_\mathrm{th}\to \infty$, the weak $(\Balpha,\delta,\mathbb{H})$ surgical flows converge in the Hausdorff sense to the level set flow. Furthermore, away from the singular set of $\mathcal{M}$ the convergence is smooth.
 \end{theorem}
 \begin{proof}
 This is an immediate consequence of Proposition \ref{convergence2} and White regularity \cite{bw2005}.
\end{proof}
 
 \section{Applications of the Surgery}
 We now apply the above surgery formalism to prove a Schoenflies type  theorem for hypersurfaces of entropy less than $\lambda(\mathbb{S}^1\times \mathbb{R}^2)$, without having to manually construct the isotopies. Such a proof was conjectured in \cite[Conjecture 1.9]{ccms21}.
 The previous best bound on the entropy was $\lambda(\mathbb{S}^2\times \mathbb{R}^1)$ and was achieved independently by Bernstein--Wang \cite{bw20} and Chodosh--Choi--Mantoulidis--Schulze \cite{ccms20}.
 
 Recall the definition of entropy for a hypersurface from \cite{CM15}.
 \begin{defn} The Entropy of a hypersurface $\Sigma$ is
 \begin{align*}
     \lambda(\Sigma)=\sup_{x_0,t_0}\left(\frac{1}{4\pi t_0}\right)^{\frac{n}{2}}\int_\Sigma \exp{\left(-\frac{|x-x_0|^2}{4t_0}\right)}d\mu,
 \end{align*}
i.e.~ the supremum of the Gaussian densities over all scales and base-points. It can be considered a measure of the complexity of an embedding.
 \end{defn}
 
We first discuss the topological consequences of surgery. Recall, from Theorem \ref{discarded} we know discarded components will be diffeomorphic to $\mathbb{S}^n$ or $\mathbb{S}^{n-1}\times \mathbb{S}^1$. Moreover, we have the following 
 
 \begin{lemma}\label{isotopy}Let $\mathcal{M}_\mathbb{H}$ be a smooth mean curvature flow with surgery from the smooth initial condition $M$. Then,
 \begin{enumerate}[(i)]
     \item The flow $\mathcal{M}_{\mathbb{H}}$ is a smooth isotopy between times of surgery.
  \item Let $\tilde{M}$ be a connected component of the timeslice $M_t$, for any $t, \ 0<t<T_\mathrm{Ext}.$ The size of the fundamental group of $\tilde{M}$ satisfies $|\pi_1(\tilde{M})|\leq |\pi_1(M)|$.
    \end{enumerate}
 \end{lemma}
 
\begin{proof}
(i) It is immediate from the definition that smooth mean curvature flow is an isotopy. The flow $\mathcal{M}_{\mathbb{H}}$ is a smooth flow with surgery, and thus a mean curvature flow between times of surgery. This proves the first statement.

(ii) From part (i), we know that any topological changes that occur must happen at surgeries. It is sufficient to show the claim at the first surgery time, as at future surgery times we can treat each connected component present before surgery as a separate flow.

 Let $t$ be the first time of surgery. We denote the pre-surgery hypersurface by $M^-_t$ and post neck-replacement, but pre-component dropping, by $M^\#_t$. Note that it is possible for $M^\#_t$ to be disconnected. By item (i), we have $\pi_1(M^-_t)=\pi_1(M)$. We need only to consider connected components of $M^\#_t$ as clearly any component present at time between the first and second times of surgery must have evolved from some connected component of $M^\#_t$.
Thus, it is sufficient to show $|\pi_1(M^-_t)|\geq|\pi_1(\tilde{M})|$, where $\tilde{M}$ is a connected component of $M^\#_t$. This follows immediately by \cite[Proposition 3.23]{hssurg}, which shows $M^-_t$ is diffeomorphic to the connected sum (reversing the neck-replacement) of the connected components of $M^\#_t$. For completeness we prove our claim directly, by showing every non-trivial element of $\pi_1(\tilde{M})$ corresponds to a non-trivial element of $\pi_1(M^-_t)$. 

Let $P_i=(p_i,t), i\in\{1,2, \ldots, N\}, N<\infty$ be the centre of each $\delta$-neck that is about to be replaced by caps at time $t$. We know all modifications are made in $B=\cup^N_i B(p_i,5\Gamma \hneck^{-1})$ (see Definition \ref{necktocap} with $s=\hneck^{-1}$). 

Let $\gamma\in\pi_1(\tilde{M})$ be a non-trivial element. We can take this element to be represented by a curve $\tilde\gamma$ lying entirely in $\tilde{M}\backslash\{\tilde{M} \cap B\}$. This follows as each connected component of $\tilde{M} \cap B$ is diffeomorphic to our standard cap. Since the cap is simply connected, any portion of curve that enters a cap is homotopic to a curve on the boundary. Morally, we can consider this curve as detecting some topology unaffected by our surgery at time $t$.

Since $\tilde{\gamma}\cap B =\emptyset$, we can consider it as a curve in $M^-_t$, since $\tilde{M}\backslash\{\tilde{M} \cap B\} \subset M^-_t$. Clearly this curve cannot represent the trivial homotopy class as the connected sum operation cannot `remove topology'. Consequently, $|\pi_1(\tilde{M})|\leq|\pi_1(M)|)$.

\end{proof}

\begin{remark}
 It is of note that the surgery procedure detailed above can break handles in two ways. This is best illustrated by the following examples.
 \begin{enumerate}
     \item Consider the 2-convex embedding of the torus known as the `wedding band'. Deform it in a 2-convex manner such that one region is a much tighter neck than other regions. This flow will develop an inward neck pinch under mean curvature flow. If surgery is performed once, we are left with a `sausage', smoothly isotopic to a sphere.
     \item Consider a sphere with small holes drilled in around the poles, that has had the ends of a cylinder attached smoothly to each hole. This is a smooth embedding of the torus. This cylinder is a long thin neck which, heuristically, one expects would form an outward neck pinch under mean curvature flow. If one were to replace this neck by surgery, the resulting hypersurface is a sphere with the poles (smoothly) pushed in. This hypersurface is smoothly isotopic to a sphere.
 \end{enumerate}
\end{remark}
 
 \begin{theorem}[Low-entropy Schoenflies for $\mathbb{R}^4$]
  Let $\Sigma^3\subset\mathbb{R}^4$ be a hypersurface homeomorphic to $\mathbb{S}^3$ with entropy $\lambda(\Sigma)\leq\lambda(\mathbb{S}^1\times \mathbb{R}^2)$. Then $M$ is smoothly isotopic to the round $\mathbb{S}^3$.
 \end{theorem}
 \begin{proof}
 $\Sigma^3\subset\mathbb{R}^4$ be a hypersurface homeomorphic to $\mathbb{S}^3$ with entropy $\lambda(\Sigma)\leq\lambda(\mathbb{S}^1\times \mathbb{R}^2)$.
 By \cite{ccms21}, there is a small (isotopic) perturbation of $\Sigma$, $\hat\Sigma$, such that the unit-regular Brakke flow, $\mathcal{M}$, emerging from $\hat\Sigma$ is unique and encounters only spherical and neck-pinch singularities.
We find $\gamma>0$ such that $\max_{\mathbf{x}\in\hat\Sigma} \{|\mathbf{A}(\mathbf{x})|\}<\gamma$ and fix $\alpha, \beta$ and $ \delta>0$ as discussed in section 5. By Theorem \ref{existence}, the parameters $H_\mathrm{trig}$ and $\Theta$ can be chosen such that there is smooth $(\Balpha,\delta,\mathbb{H})$-flow with surgery $\mathcal{M}_{\mathbb{H}}$ that approximates the flow $\mathcal{M}$. In addition, we suppose $H_\textrm{th}$ and $\Theta$ are large enough that the conclusion of Theorem \ref{dropped} holds.

It remains to show that all the dropped components of $\mathcal{M}_{\mathbb{H}}$ are not tori and no handles are broken.

\begin{claim}
 The topological constraint that $\Sigma$ is homeomorphic to $\mathbb{S}^3$ rules out \begin{enumerate}[(a)]
     \item Dropped components being diffeomorphic to tori, $\mathbb{S}^2\times \mathbb{S}^1$.
     \item The breaking of a handle during surgery.
 \end{enumerate} 
\end{claim}\begin{proof}
We prove (a), (b) follows identically.
Suppose for contradiction that there is at least one dropped component that is a torus. Let $t$ be the first time a torus is dropped in surgery. It is clear that some component of the pre-surgery hypersurface $M_{t^-}$ would have a non-trivial fundamental group (i.e. size greater than 1).
 By Lemma \ref{isotopy}, the initial condition $\hat\Sigma$ must also have had non-trivial fundamental group. This is a contradiction to $\Sigma$ being homeomorphic to $\mathbb{S}^3$.
 \end{proof}
Thus, all dropped components are isotopic to spheres and no handles are broken.

We now use backward induction to deduce $\hat\Sigma$ is smoothly isotopic to the round $\mathbb{S}^3$. There are finitely many surgeries, thus, there is a finite set of times $t_1<\ldots<t_n$ when the flow is stopped.

Observe, at $t^-_n$, the final non-empty time slice of $\mathcal{M}_{\mathbb{H}}$, we have a collection of 2-convex components diffeomorphic to spheres. Each connected component is smoothly isotopic to a sphere. (Such an isotopy can be found in \cite{bhh}.) Following the flow back to the $(n-1)^\mathrm{th}$ surgery, item (i) of Lemma \ref{isotopy} shows each connected component of the $t^+_{n-1}$ time slice is smoothly isotopic to spheres. Reversing the surgery, the $t^-_{n-1}$ time-slice is obtained by connecting the components present in $t^\#_{n-1}$) with smooth necks. Explicitly, we have the connected components present in $t^+_{n-1}$ and a collection of dropped components.

Claim 6.1 shows that these dropped components are diffeomorphic to spheres. No handles will be introduced when we reverse the surgery. Thus, the reversing of the surgery is a connected sum of spheres. In particular, $t^-_{n-1}$ is smoothly isotopic to some sub-collection of the connected components, and thus isotopic to a collection of round $\mathbb{S}^3$.

By reverse induction, this is true for the initial time-slice. Since there is only one connected component, the hypersurface $\hat\Sigma$ is smoothly isotopic to the round $\mathbb{S}^3$.
\end{proof}
\appendix
\section{Graphicality and Pseudolocality}

\begin{theorem}[Interior estimates for Graphs \cite{EHest}]\label{ehest}
Let $M^n\subset\mathbb{R}^{n+1}$ be a smooth hypersurface.
Let $R>0$ be such that $M_t$ can be written as a graph over $B^n_R$, an $n$-ball of radius $R$ in some hyperplane, for $t\in[0,T]$.
Suppose further that the gradient is bounded, i.e.~ we denote the graph function by $u$ and for each $t\in[0,T]$ we have
\begin{align*}
    \sqrt{1+|Du_t|^2}\leq 1+\eta
\end{align*}
Where $\eta>0$ depends only on the dimension. Then, for any $t\in[0,T]$ and $\theta \in (0,1)$, we have
\begin{align*}
    \sup_{B_{\theta R (y_0)}\times [0,T]}|A|^2 \leq C(n,\theta,R)\sup_{B_{ R} (y_0)\times \{0\}}|A|^2
\end{align*}
\end{theorem}

This is immediate from the Theorem 3.1 of \cite{EHest} under the assumption of bounded initial curvature. See also \cite{begleymoore}, where the estimates are established for high co-dimension.

We also state the pseudolocality result of Ilmanen--Neves--Schulze, in the co-dimension 1 case. We also don't require bounded area ratios as we only care about the local case, hence can rely on the local Monotonicity formula. See also the pseudolocality result stated in \cite{ChenYin}.

Let $x\in\mathbb{R}^{n+1}, x=(\hat x,\tilde x)$.
We define the cylinder $\mathcal{C}_r(x_0)\subset \mathbb{R}^{n+1}$  by \begin{align*}
    \mathcal{C}_r(x)=\{x\in\mathbb{R}^{n+1} \mathrm{s.t.} |\hat x-\hat x_0|<r, |\tilde{x}-\tilde{x}_0|<r\}
\end{align*}
\begin{theorem}[Pseudolocality \cite{pseudo}]\label{pseudo}
Let $\{M_t\}_{t\in[0, T)}$ be a smooth mean curvature flow of embedded hypersurfaces in $\mathbb{R}^{n+1}$. Then, for any $\eta>0$ there exists $\varepsilon,\vartheta>0$ depending only on $n,\eta$ such that if $x_0\in M_0$ and $M_0\cap \mathcal{C}_1(x_0)$ can be written as $\mathrm{graph}(u)$, where $u: B^n(x_0)\to\mathbb{R}$ with Lipschitz constant less than $\varepsilon$, then
\begin{align*}
    M_t\cap \mathcal{C}_{\vartheta}(x_0), t\in[0,\vartheta^2)\cap[0,T)
\end{align*}
is a graph over $B^n_{\vartheta}(x_0)$ with Lipschitz constant less than $\eta$ and height bounded by $\eta\vartheta$.
\end{theorem}
\begin{remark}
As is remarked in \cite[Remarks 1.6]{pseudo}, the above statement holds with only the presumption that $\{M_t\}_{t\in[0, T)}$ is a unit-regular integral Brakke flow. The proof requires the use of Brakke's local regularity theorem, \cite{Brakke} in place of White's local regularity, \cite{bw2005}.
\end{remark}

\begin{theorem}\label{bddctrl}
Let $\mathcal{M}'$ be a $(\alpha,\delta)$-Brakke flow. Suppose $X=(\textbf{x},t_x)\in \mathcal{M}\cap\Omega_{(\alpha,\delta)}$ with $t_x\leq t_F$, where $t_F$ is the final time surgeries are performed. Then, $\mathcal{M}\cap P(X,\xi |H(X)|^{-1})$ is a smooth $(\alpha,\delta)$-flow in the sense of Haslhofer--Kleiner.
\end{theorem}
\begin{proof}
Note, we do not need to check $Y \in P(X,\xi |H(X)|^{-1})\cap \Omega_{(\alpha,\beta)}$, by our strict definitions of how and when surgery is performed. Since no surgeries occur outside of $\Omega_{(\alpha,\delta)}$ it is sufficient to check the flow is $\beta$-uniformly 2-convex and $\alpha$-noncollapsed.

Suppose $X\in\Omega_{(\alpha,\beta)}$ and $Y=(\mathbf{y},t_y)\in P(X,\xi |H(X)|^{-1})\cap \Omega_{(\alpha,\beta)}^c\neq \emptyset$.  From the definition of a backwards parabolic cylinder, we have that $\mathbf{y}\in B(\bf{x},\xi|H(X)|^{-1} )$. Let $L$ be the line segment joining $\mathbf{x}$ to $\mathbf{y}$ in the timeslice $\mathbb{R}^{n+1}\times\{t_y\}$. This line segment must pass through $\partial \Omega_{(\alpha,\delta)}$. Let $Z=(\mathbf{z},t_y)$ denote the point on $L$ intersecting $\partial \Omega_{(\alpha,\delta)}$. Clearly we have $|\mathbf{z}-\mathbf{y}|<|\mathbf{x}-\mathbf{y}|\leq \xi |H(X)|^{-1}$. By the maximum principle, we have $|H(Z)|\leq|H(X)|$, and so $Y\in P(Z, \xi |H(Z)|^{-1})$). By the assumption $t_x\leq t_f$, we know that at $t=t_y$, the flow $\mathcal{M}'$ remains $\delta$-graphical over $\mathcal{M}$ in the neighbourhood of the boundary $N$. By the definition of $N$, Definition \ref{bddnbhd}, we have $P(Z, \xi |H(Z)|^{-1})\subset N$. In particular, by our choice of $\delta$, at the point $Y\in\mathcal{M}'$, the flow is $\beta$-uniformly 2-convex and $\alpha$-noncollapsed. 
\end{proof}

 \bibliographystyle{alpha}
 \bibliography{surgery}
\end{document}